\documentclass[12pt]{amsart}   	% use "amsart" instead of "article" for AMSLaTeX format
\usepackage{geometry}                		% See geometry.pdf to learn the layout options. There are lots.
\usepackage{graphicx}				% Use pdf, png, jpg, or epsÂ§ with pdflatex; use eps in DVI mode
\usepackage{amssymb}
\usepackage{mathtools}
\usepackage[new]{old-arrows}
\usepackage{stmaryrd}

\usepackage[shortlabels]{enumitem}

\usepackage{amsthm,amssymb,amsmath,amsfonts,extarrows,bm}
\usepackage[all,cmtip,2cell]{xy}

\usepackage[colorlinks=true,backref=page]{hyperref}
\usepackage[noabbrev,nameinlink]{cleveref}

\usepackage{xfrac,faktor}
\usepackage{multirow}

\usepackage[dvipsnames]{xcolor}

\usepackage{comment}
\usepackage{tikz-cd}
\usepackage[textsize=tiny]{todonotes}

\renewcommand*{\backref}[1]{}
\renewcommand*{\backrefalt}[4]{%
    \ifcase #1 %
        (Not cited.)%
    \or
        $\uparrow$ #2%
    \else
        $\uparrow$ #2%
    \fi
}

\newtheorem{theorem}{Theorem}[section]
\newtheorem{lemma}[theorem]{Lemma}

\newtheorem{conjecture}[theorem]{Conjecture}
\newtheorem{conj}[theorem]{Conjecture}
\newtheorem{prop}[theorem]{Proposition}

\theoremstyle{definition}
\newtheorem{defn}[theorem]{Definition}

\newtheorem{remark}[theorem]{Remark}

\numberwithin{equation}{section}

\newcommand{\overbar}[1]{\mkern 1.5mu\overline{\mkern-1.5mu#1\mkern-1.5mu}\mkern 1.5mu}
\newcommand{\AAA}{\mathbb{A}}
\newcommand{\CC}{\mathbb{C}}

\newcommand{\QQ}{\mathbb Q}
\newcommand{\Qbar}{\overline{\mathbb Q}}

\newcommand{\RR}{\mathbb{R}}

\newcommand{\ZZ}{\mathbb{Z}}
\newcommand{\bA}{\mathbf{A}}

\newcommand{\calA}{\mathcal{A}}
\newcommand{\calB}{\mathcal{B}}

\newcommand{\calD}{\mathcal{D}}

\newcommand{\calF}{\mathcal{F}}
\newcommand{\calG}{\mathcal{G}}
\newcommand{\calH}{\mathcal{H}}

\newcommand{\calL}{\mathcal{L}}
\newcommand{\calM}{\mathcal{M}}

\newcommand{\calO}{\mathcal{O}}
\newcommand{\calP}{\mathcal{P}}

\newcommand{\calS}{\mathcal{S}}
\newcommand{\calU}{\mathcal{U}}
\newcommand{\calV}{\mathcal{V}}
\newcommand{\calW}{\mathcal{W}}
\newcommand{\calX}{\mathcal{X}}
\newcommand{\calY}{\mathcal{Y}}
\newcommand{\calZ}{\mathcal{Z}}
\newcommand{\gothA}{\mathfrak{A}}

\newcommand{\gothb}{\mathfrak{b}}

\newcommand{\gothF}{\mathfrak{F}}
\newcommand{\gothH}{\mathfrak{H}}

\newcommand{\dual}{\vee}

\newcommand{\ws}{\mathrm{ws}}

\DeclareMathOperator{\an}{an}

\DeclareMathOperator{\biZar}{biZar}

\DeclareMathOperator{\codim}{codim}
\DeclareMathOperator{\covol}{covol}

\DeclareMathOperator{\End}{End}

\DeclareMathOperator{\image}{im}
\DeclareMathOperator{\id}{id}

\DeclareMathOperator{\Lie}{Lie}

\DeclareMathOperator{\Mat}{Mat}

\DeclareMathOperator{\ord}{ord}

\DeclareMathOperator{\pr}{pr}

\DeclareMathOperator{\rdim}{rel.dim}

\DeclareMathOperator{\Sp}{Sp}

\DeclareMathOperator{\spec}{sp}

\DeclareMathOperator{\Sym}{Sym}

\DeclareMathOperator{\Tr}{Tr}

\begin{document}

\title[Degeneracy loci and unlikely intersections in abelian schemes]{Degeneracy loci and unlikely intersections \\ in abelian schemes}

\author{Fabrizio Barroero}
\address[Fabrizio Barroero]{Department of Mathematics and Physics, Roma Tre University \\
Largo San Leonardo Murialdo 1, 00146 Rome, Italy}
\email{fabrizio.barroero@uniroma3.it}

\author{Laura Capuano}
\address[Laura Capuano]{Department of Mathematics and Physics, Roma Tre University \\
Largo San Leonardo Murialdo 1, 00146 Rome, Italy}
\email{laura.capuano@uniroma3.it}

\author{Tangli Ge}
\address[Tangli Ge]{Department of Mathematics and Statistics, Virginia Commonwealth University, 1015 Floyd Avenue Richmond, Virginia 23284, US}
\email{get2@vcu.edu}

\author{Francesco Tropeano}
\address[Francesco Tropeano]{Institute for Theoretical Sciences, Westlake University, No. 600 Dunyu Road, Sandun town, Xihu district, 310030 Hangzhou, Zhejiang Province, China}
\email{francesco.tropeano@westlake.edu.cn}

\date{\today}
%  algebraic theory of abelian varieties, subvarieties of abelian vairieties, heights 
%abelian varieties of dimension >1 (11G05), don't know if you want to add it

%\thanks{1}

\begin{abstract}
Let $\calA\to S$ be an abelian scheme over a normal quasi-projective
variety over a number field, and let $\calX\subseteq\calA$ be a
positive-dimensional subvariety. For every integer $t>0$, we prove
that, outside the $t$-degeneracy locus of $\calX$, at most finitely many points lie in flat group subschemes of relative dimension less
than $t$. In particular, this intersection is not Zariski dense whenever $\calX$ is $t$-nondegenerate. The proof combines a bounded height theorem with uniform large Galois orbit estimates, o-minimal point counting, and mixed Ax--Schanuel, following the Pila--Zannier strategy. The arithmetic input is an explicit bound for the complexity of endomorphism relations, obtained from geometry of numbers and height estimates for abelian varieties. As applications, we formulate a relative Mordell--Lang conjecture, prove it for section images in powers of certain abelian schemes, and establish injectivity of specialization of the group of sections outside a proper closed subset under variation and dimension
hypotheses.
\end{abstract}   

\maketitle

\section{Introduction}

The Zilber--Pink conjecture predicts that intersections of a subvariety with special subvarieties of sufficiently large codimension are not Zariski dense, unless the geometry forces otherwise. Originating in work of Bombieri, Masser and Zannier~\cite{BMZ_99} and Zilber~\cite{Zilber_exp} for tori, and formulated by Pink~\cite{Pink_comb_unpublished} in the setting of mixed Shimura varieties, it provides a common framework for the Manin--Mumford, Mordell--Lang and Andr\'e--Oort problems (see the books of Pila~\cite{Pila_book} and Zannier~\cite{Zannier_book} for further background on the subject).

In this paper, we study a conjecture formulated by Pink in the context of abelian schemes. Let $S$ be a normal quasi-projective variety and let $\pi:\calA\to S$ be an abelian scheme of relative dimension $g$, both defined over a field $K\subseteq\CC$. For an integer $t$, denote by $\calA_{(<t)}$ the union of all algebraic subgroups of dimension less than $t$ in the geometric fibers of $\calA/S$. In this setting, the prediction of \cite[Conj.~6.1]{Pink_comb_unpublished} takes the following form.

\begin{conj}[Pink]\label{conj:Pink}
Let $\calX\subseteq\calA$ be a subvariety of dimension $d$ dominating $S$, with geometrically irreducible generic fiber. If $\calX$ is not contained in any strict group subscheme of $\calA$, then $\calX\cap\calA_{(<g-d)}$ is not Zariski dense in $\calX$.
\end{conj}

% Over a number field, we prove a finiteness result for intersections with flat group subschemes, valid for subvarieties of arbitrary dimension: outside Gao's $t$-degeneracy locus, only finitely many points lie in flat group subschemes of relative dimension less than $t$.

Although \Cref{conj:Pink} remains open in general, several cases are known. When $d=g-1$, the intersection concerns fiberwise torsion points. This case follows from the \emph{relative Manin--Mumford conjecture}, initially investigated in a series of papers by Masser and Zannier~\cite{MasserZannier_torsiononproductsofEC,MasserZannier_torsionEC1,MasserZannier_torsionEC2,MasserZannier_Torsionpoints,MasserZannier_torsion20} and proved in full generality by Gao and Habegger~\cite{GH_RMM}. When $S$ and $\calX$ are curves and $\calA$ is a fibered product of elliptic schemes, the conjecture follows from work of the first- and second-named authors \cite{BC_PowersofEC,BC_productsofEC}, the first-named author \cite{B_powers}, and the first-named author and Dill \cite{BD_distinguished}. 

For general abelian schemes over a curve, defined over a number field, work of Habegger--Pila \cite{HP_ominimality} and the first- and second-named authors \cite{BC_CurvesinAbelianScheme} establishes the corresponding result for intersections of curves with flat group subschemes of relative codimension at least $2$; see also \cite{Capuano_survey} for a survey. Our result extends this study to higher-dimensional subvarieties and bases.

From now on, assume that $K$ is a number field. For an integer $t$, write
\[
\calA_{<t}:=\bigcup_{\rdim\calH<t}\calH,\quad \calA_{\le t}:=\bigcup_{\rdim\calH\le t}\calH
\]
where both unions run over all flat group subschemes $\calH\subseteq\calA$. In particular, $\calA_{<t}\subseteq\calA_{(<t)}$.

The exceptional set in our theorem is the $t$-\emph{degeneracy locus}, introduced by Gao \cite{Gao_Betti} to study the rank of the Betti map and also used by Gao and Habegger in their proof of the relative Manin--Mumford conjecture. For an irreducible closed subvariety $\calY\subseteq\calA$, let $\calY^{\textnormal{ws}}$ denote its relative weakly special closure, as defined in \S3.2. Its definition uses abelian subvarieties, torsion, and the constant part over the generic point of $\pi(\calY)$. Define the vertical defect by
\[
\delta_v(\calY):=\rdim\calY^{\textnormal{ws}}-\dim\calY.
\]
For an irreducible closed subvariety $\calX\subseteq\calA$ and an integer $t$, put
\[
\calX^{\textnormal{deg}}(t):= \bigcup_{\substack{\calY\subseteq\calX\ \mathrm{irreducible}\\
                   \dim\calY>0,\ \delta_v(\calY)<t}}\calY.
\]
This union is Zariski closed; see \cite{Gao_Betti,Ge_IntersectI}. We say that $\calX$ is \emph{$t$-nondegenerate} if $\calX^{\textnormal{deg}}(t)\ne\calX$. Our main result is the following.

\begin{theorem}[\Cref{thm:main}]\label{thm:mainIntro}
Let $S$ be a normal quasi-projective variety and let $\calA/S$ be an abelian scheme, both defined over a number field $K$. Let $\calX\subseteq\calA$ be a positive-dimensional subvariety defined over $K$. For every integer $t>0$, the intersection
\[
\bigl(\calX\setminus\calX^{\textnormal{deg}}(t)\bigr) \cap\calA_{<t}
\]
is finite. In particular, if $\calX$ is $t$-nondegenerate, then $\calX\cap\calA_{<t}$ is not Zariski dense in $\calX$.
\end{theorem}

Positive-dimensional intersections must be excluded: every positive-dimensional irreducible component $\calY$ of $\calX\cap\calH$, with $\calH$ flat of relative dimension less than $t$, satisfies
\[
\delta_v(\calY)\leq\rdim\calH-\dim\calY<t.
\]
The theorem says that, after removing the degeneracy locus, only finitely many points remain even as $\calH$ varies through all such group subschemes. No simplicity or variation assumption on $\calA/S$ is required. 

Suppose that $\dim\calX<g$ and that $\calX$ is $(g-\dim\calX)$-nondegenerate. Taking $t=g-\dim\calX$, our theorem gives nondensity of the intersection with flat group subschemes of relative codimension greater than $\dim\calX$. This is the codimension range predicted by Pink's \Cref{conj:Pink}, restricted to flat group subschemes.

\medskip

Note that \Cref{thm:mainIntro} recovers several known finiteness results:
\begin{itemize}
    \item
    When $S$ is a point, the theorem recovers the finiteness statement of Habegger--Pila~\cite[Thm.~9.15(iii)]{HP_ominimality} and its nondensity consequence in part~(iv). For curves in a fixed abelian variety, however, our nondegeneracy hypothesis requires that the curve not be contained in any translate of a proper abelian subvariety. This is stronger than the hypothesis of \cite[Thm.~1.1]{HP_ominimality}, so our theorem does not recover that result in full.

    \item Suppose that $g\geq2$ and that $\calX$ is a curve. Taking $t=g-1$ gives finiteness of its intersection with flat group subschemes of relative dimension at most $g-2$, provided that $\calX$ is $(g-1)$-nondegenerate.

    If $S$ is a curve, $\calX$ dominates $S$, and $\calA/S$ has no fixed part, the geometric hypothesis that $\calX$ is not contained in a strict group subscheme even after a finite base change implies the nondegeneracy, by \Cref{defn:t-deg} and \Cref{thm:Gao_criterion}. Thus our theorem recovers \cite[Thm.~1.1]{BC_CurvesinAbelianScheme} in this case.

    \item For $t=1$, the locus $\calA_{<1}$ is the fiberwise torsion locus $\calA_{\mathrm{tors}}$. Our theorem gives finiteness of torsion points outside $\calX^{\textnormal{deg}}(1)$ and therefore implies the degeneracy statement for the universal family in \cite[Thm.~3.4]{GH_RMM}. Together with the induction argument in \cite[Cor.~6.3]{GH_Degeneracy} and the usual reductions to the universal family, this yields the relative Manin--Mumford conjecture over $\Qbar$.
\end{itemize}

%%% 

\subsection{The bounded height input}

A key arithmetic input in the proof of \Cref{thm:mainIntro} is the bounded height theorem of the third-named author \cite[Thm.~8.4]{Ge_IntersectI}, which we recall below.

Bounded height theorems typically treat intersections at complementary dimension, where the expected intersection dimension is zero. These are \emph{likely intersections}, and subgroups of strictly larger codimension belong to the \emph{unlikely intersections} range. In general, a uniform height bound requires removing a geometric exceptional locus.

Bombieri, Masser and Zannier~\cite{BMZ_Anomalous} formulated the bounded height conjecture for algebraic tori, proving it in the case of curves. Habegger proved then the conjecture its analogue for abelian varieties~\cite{Habegger_BHCtori,Habegger_BHConAV}: outside the anomalous locus of a subvariety $X$, the points lying in algebraic subgroups of codimension at least $\dim X$ have bounded height. Habegger and Pila combined this height bound with their counting method to obtain their unlikely intersection results discussed above.

In the cited curve results of Masser--Zannier and the first- and second-named authors, Silverman's specialization theorem \cite[Thm.~C]{Silverman_Specialization} provides the required height bounds. In Gao--Habegger's proof of relative Manin--Mumford, the height bound comes from the height inequality in their joint work with Dimitrov \cite{DimitrovGaoHabegger_UML}.

The result of the fourth-named author \cite{Ge_IntersectI} extends this bounded height approach to abelian schemes, with the degeneracy locus as the exceptional set. We will use the following form of that result. Fix a relatively ample line bundle on $\calA$ and a Weil height associated with a projective embedding of $S$; the \emph{total height} is the sum of the associated fiberwise N\'eron--Tate height and the Weil height on the base.

\begin{theorem}[\cite{Ge_IntersectI}]\label{thm:totHeight}
Let $S$ be a normal quasi-projective variety over a number field $K$, and let $\pi:\calA\to S$ be an abelian scheme. Let $t\in\ZZ_{\ge0}$ and let $\calX$ be a $t$-nondegenerate subvariety of $\calA$, defined over $K$. Then
\[
\bigl(\calX\setminus\calX^{\textnormal{deg}}(t)\bigr) \cap\calA_{\le t}
\]
has bounded total height.
\end{theorem}

The result of \cite[Thm.~8.4]{Ge_IntersectI} also applies to points that are close, in a suitable height sense, to the group subschemes in question. When $S$ is a point, \Cref{thm:totHeight} recovers Habegger's bounded height theorem \cite{Habegger_BHConAV}. When $S$ and $\calX$ are curves, with $\calX$ dominating $S$, the corresponding height bound follows from Silverman's specialization theorem \cite{Silverman_Specialization} under its no-fixed-part hypothesis.

The cited theorem gives bounded total height for the intersection with $\calA_{\le t}$. We use this bound as the starting point for the proof of \Cref{thm:mainIntro}, which establishes finiteness for the intersection with $\calA_{<t}$, outside the same degeneracy locus.

%%%%%%
\subsection{Proof outline}
Our proof follows the Pila--Zannier strategy~\cite{Pila_Zannier}, which combines o-minimal point counting~\cite{Pila_Wilkie} with Diophantine estimates. We begin with the total height bound supplied by \Cref{thm:totHeight} on the relevant intersection outside the degeneracy locus. Bounded height alone does not imply finiteness, since the points may have arbitrarily large degree over $K$. We address this by relating their field degrees to the complexity of the subgroup relations they satisfy.

Fix $\Gamma=\End(\calA/S)$ and regard it as a subring of the endomorphism ring of each fiber. For $x\in\calA_s(\bar K)$, we call its $\Gamma$-special closure the common kernel of the endomorphisms in $\Gamma$ that annihilate $x$. Its degree, denoted by $\Delta_\Gamma(x)$, accounts for the degree of the identity component and the number of connected components. In \Cref{thm:LGO} we give an explicit upper bound for this complexity in terms of field degree, heights, and the endomorphism lattice. The argument combines geometry of numbers, following Masser \cite{Masser_Linear}, with estimates of Gaudron--R\'emond \cite{GR}. For a fixed family and a fixed bound on total height, it yields a large-Galois-orbit result with constants $c,\epsilon>0$ such that
\[
[K(x):\QQ]\geq c\Delta_\Gamma(x)^\epsilon
\]
see \Cref{prop:LGO}.

In Betti coordinates, subgroup relations can be encoded by rational data whose height is bounded polynomially in $\Delta_\Gamma(x)$. If the complexity of a point of the intersection is large, the Galois orbit supplies enough such data to apply the Habegger--Pila refinement~\cite[Cor.~7.2]{HP_ominimality} of the Pila--Wilkie counting theorem~\cite{Pila_Wilkie} to produce a nonconstant analytic arc with the required semialgebraic property. Mixed Ax--Schanuel~\cite{Gao_mixed}, in the form of \Cref{prop:transcendence}, implies that the arc belongs in the degeneracy locus, contradicting the choice of the point. Complexity is therefore bounded, leaving only finitely many relevant flat group subschemes to consider. Their positive-dimensional intersections with $\calX$ are already contained in the degeneracy locus, which proves finiteness.

\subsection{Applications}
In \S7 we give two applications concerning finitely generated groups of sections. The first is a relative Mordell--Lang problem. For a finitely generated subgroup $\Lambda\subseteq\calA(S)$, denote by $\Lambda_s^{\mathrm{div}}$ the divisible hull of its specialization in $\calA_s(\bar K)$. We propose a \emph{relative Mordell--Lang conjecture} governing the locus
\[
\bigcup_{s\in S(\bar K)} \bigl(\calX_s(\bar K)\cap\Lambda_s^{\mathrm{div}}\bigr).
\]
Suppose that $\calX$ dominates $S$, has geometrically irreducible generic fiber, and satisfies the transversality condition that $\calX_{\bar\eta}$ is not contained in any translate of a proper abelian subvariety by a point of $\Lambda_{\bar\eta}^{\mathrm{div}}$. The conjecture predicts nondensity whenever $\dim\calX<\rdim\calA$. We show that it follows from Pink's conjecture. Under the same transversality and dimension conditions, we prove it for section images in powers of an abelian scheme for which every nontrivial simple factor of the geometric generic fiber has maximal variation. Here maximal variation means that the moduli map is generically finite onto its image. The proof encodes the generators of $\Lambda$ as additional coordinates and applies our main theorem to the resulting subgroup relations.

The second application concerns injectivity of specialization of $\Lambda$. Suppose that every nontrivial geometric simple factor of the generic fiber of $\calA/S$ has maximal variation and dimension greater than $\dim S$. We prove that
\[
\operatorname{sp}_s:\Lambda\longrightarrow\calA_s(\bar K)
\]
is injective for all $s\in S(\bar K)$ outside a proper Zariski closed subset. In particular, over a curve the exceptional set is finite under these hypotheses. This complements bounded height results for specialization \cite{Silverman_Specialization,Ge_IntersectI} by establishing Zariski nondensity under suitable assumptions on the variation and dimensions of the geometric simple factors.

The paper is organized as follows. Section~2 develops the endomorphism and lattice estimates used later. Section~3 reviews weakly special subvarieties, degeneracy loci, and Betti coordinates. Section~4 supplies the counting and functional transcendence arguments, while Section~5 proves the complexity and large Galois orbit estimates. We prove the main theorem in Section~6 and give the applications in Section~7.

\subsection*{Notation and conventions}
\begin{enumerate}
    \item Unless otherwise stated, $K\subseteq\CC$ is a number field containing $\QQ[\zeta_3]$,  and $\bar K$ denotes its algebraic closure in $\CC$. Let $(A,L)$ be a $g$-dimensional polarized abelian variety over $K$, with polarization given by the ample line bundle $L$.
    
    \item Varieties are geometrically irreducible and reduced. Unless otherwise specified, subvarieties are Zariski closed. We write $S$ for a normal quasi-projective variety over either $K$ or $\CC$, and $\calA/S$ for an abelian scheme over $S$.
    
    \item  We denote by $\hat h_L:A(\bar K)\to\RR_{\ge0}$ the quadratic part of the N\'eron--Tate height associated to $L$, namely the canonical height associated to the symmetric class $\frac12(L+[-1]^*L)$. By abuse of notation, we call this the N\'eron--Tate height associated to $L$. We denote by $h_F(A)$ the stable Faltings height of $A$.

    \item The ring $\End(A)$ consists of the endomorphisms of $A$ \emph{defined over $K$}. For $R=\QQ$ or $\RR$, we write $\End_R(A):=\End(A)\otimes_\ZZ R$ for the
    corresponding endomorphism algebra.

    \item For a family $f:\calX\to\calY$ and a subscheme $\calZ\subseteq\calX$, we write
    \[
    \rdim\calZ=\rdim_f\calZ:=\dim\calZ-\dim f(\calZ)
    \]
    for the dimension of $\calZ$ relative to the base.
\end{enumerate}

\subsection*{Acknowledgements}
The authors are grateful to Ziyang Gao for introducing L.C. and T.G. to each other at the Mordell 100 conference, where they began discussing the problem that eventually led to this work.\\
F.B., L.C.~and F.T.~were supported by the PRIN 2022 project ``2022\-HPSNCR: Semi-abelian varieties, Galois representations and related Diophantine problems'' and the ``National Group for Algebraic and Geometric Structures, and their Applications'' (GNSAGA INdAM). F.B.~and L.C.~ were also supported by the project SERICS (PE00000014)
under the MUR National Recovery and Resilience Plan funded by the
European Union - NextGenerationEU. T.G.~thanks the Department of Mathematics at Princeton University for its hospitality during his instructorship. F.T.~thanks the Department of Mathematics at Princeton University for its hospitality and support during the preparation of this paper.

\subsection*{Disclosure of AI assistance} 
The authors began this project in early 2025 and developed its main mathematical ideas without AI assistance. ChatGPT-6 Astra was subsequently used for language editing and assistance with proof checking. A suggestion from ChatGPT-6 Astra led to an improvement of the bound in Theorem~\ref{thm:LGO} through Lemma~\ref{lem:relations}, and to the application in \Cref{cor:rml-section-images}. The authors prepared the final manuscript and take full responsibility for its content and mathematical correctness.

%\section{Introduction}
%\input{S1:Introduction}

% \section{Preliminary}
% \input{s2}

% \section{Mix}
% \input{rational_representation}

\section{Endomorphisms}
In this section, let $A$ be a complex abelian variety of dimension $g\ge 1$. We collect some preliminary results about its endomorphisms that will be used below.

\subsection{The rational representation}\label{subsec:rationalRepresentation}
The analytification of $A$ is a complex torus, which is canonically isomorphic to $V/\Lambda$, where $V=\Lie A=(H^0(A,\Omega^1))^\dual$ is the Lie algebra and $\Lambda=H_1(A,\ZZ)$. The inclusion of the period lattice identifies $\Lambda_{\mathbb R}:=\Lambda\otimes_{\mathbb Z}\mathbb R$ with the underlying real vector space of $V$. An endomorphism $\gamma:A \to A$ lifts uniquely to a $\CC$-linear map $\mathrm{d}\gamma:V \to V$ satisfying the property that $\mathrm{d}\gamma(\Lambda) \subseteq \Lambda$. This gives rise to an injective homomorphism of algebras by functoriality
\[
\rho_a: \End(A) \hookrightarrow \End_\CC(V), \qquad \rho_a(\gamma)=\mathrm{d}\gamma,
\]
called the \emph{analytic representation} of $\End(A)$. Meanwhile, the homomorphism $\gamma$ induces a $\ZZ$-linear map $\gamma_*:\Lambda \to \Lambda$ that determines $\gamma$ and $\rho_a(\gamma)$ completely. This gives rise to an injective homomorphism of algebras
\[
\rho_r: \End(A) \hookrightarrow \End_\ZZ(\Lambda), \qquad \rho_r(\gamma)=\gamma_*,
\]
called the \emph{rational representation} of $\End(A)$. We use the same notation for the scalar extensions of these representations whenever the coefficient field is clear, e.g. the $\QQ$-extensions $\End_\QQ(A)\rightarrow \End(V)$ and $\End_\QQ(A)\rightarrow \End(\Lambda_\QQ)$ are also denoted by $\rho_a$ and  $\rho_r$ respectively. 

\begin{remark}\label{rem:extendedRationalRepr}
    Fix $\gamma \in \End(A)$. Choose a complex basis of $V$ and a $\mathbb Z$-basis of $\Lambda$, and denote by $M_a(\gamma) \in \Mat_g(\CC)$ and $M_r(\gamma) \in \Mat_{2g}(\ZZ)$ the matrices representing $\rho_a(\gamma)$ and $\rho_r(\gamma)$, respectively. Let $b:V \to \RR^{2g}$ be the trivialization induced by associating to every vector its real coordinates with respect to the basis of $\Lambda$. The composition 
    \[ 
    \RR^{2g}\xrightarrow{b^{-1}} V\xrightarrow{\simeq}\CC^g 
    \]
    is represented by the \emph{period matrix} $\Pi \in \textnormal{Mat}_{g\times {2g}}(\CC)$, whose columns are the complex coordinates of the basis vectors of $\Lambda$. The representations give rise to a commutative diagram
    \[
    \begin{tikzcd}
        \RR^{2g} \arrow[r, "M_r(\gamma)"] & \RR^{2g}\\
        V \arrow[d,"\simeq", left] \arrow[u, "b"] \arrow{r}{\rho_a(\gamma)} & V \arrow[u, "b"] \arrow[d,"\simeq"]\\
       \CC^g \arrow[r, "M_a(\gamma)"] & \CC^g
    \end{tikzcd}.
    \]It leads to the following compatibility condition between the analytic and rational representations:
    \begin{equation}\label{eqn:compatibility}
    \Pi\cdot M_r(\gamma) = M_a(\gamma) \cdot \Pi.
    \end{equation}
\end{remark}

The rational representation converts the geometry of abelian subvarieties into linear algebra over $\QQ$. For an abelian subvariety $B\subseteq A$ we write $V_B:=\Lie B\subseteq V$ and $\Lambda_B:=H_1(B,\mathbb{Z})\subseteq\Lambda$. Given $\gamma\in \End(A)$, denote its image, kernel and identity component of the kernel respectively by $\image \gamma$, $\ker\gamma$, $\ker^0\gamma$. In general, for $\gamma\in \End_\QQ(A)$, we let $\image \gamma := \image(n\gamma)$ and $\ker^0 \gamma := \ker^0 (n\gamma)$, where $n\in \ZZ$ is a denominator of $\gamma$, so that $n\gamma\in \End(A)$. These definitions are independent of the chosen denominator, since multiplication by a nonzero integer does not change the image or the connected kernel of a homomorphism of abelian varieties. We record the dictionary we shall use throughout; see \cite[\S1.2, \S5.3]{BirLan}.

\begin{lemma}\label{lem:dictionary}
Let $\gamma\in\End_{\mathbb{Q}}(A)$ and let $B\subseteq A$ be an abelian subvariety.  Then:
\begin{enumerate}
  \item $\Lambda_{\image\gamma,\mathbb{Q}}
        =\rho_r(\gamma)\Lambda_{\mathbb{Q}}$ and $\Lambda_{\ker^{0}\gamma,\mathbb{Q}}=\ker\rho_r(\gamma)$; in
        particular, $\dim\image\gamma+\dim\ker^{0}\gamma=g$.
  \item $B\mapsto\Lambda_{B,\mathbb{Q}}$ is an inclusion-preserving
        bijection between the abelian subvarieties of $A$ and the
        $\mathbb{Q}$-subspaces $W\subseteq\Lambda_{\mathbb{Q}}$ such that
        $W_\RR$ is a $\mathbb{C}$-subspace of $V$.
  \item If $\gamma\in\End(A)$ is an isogeny, then
        $\deg\gamma=[\Lambda:\rho_r(\gamma)\Lambda]
        =\lvert\det\rho_r(\gamma)\rvert$.
\end{enumerate}
\end{lemma}

The subspaces in \Cref{lem:dictionary}(1) satisfy the condition of (2) automatically, since $\rho_a(\gamma)=\mathrm{d}\gamma$ is $\mathbb{C}$-linear.

\begin{defn}\label{def:charpoly}
The \emph{characteristic polynomial} and the \emph{trace} of $\gamma\in\End_{\QQ}(A)$ are defined as
\[
P_\gamma(t):=\det\bigl(t-\rho_r(\gamma)\bigr), \qquad \Tr(\gamma):=\Tr\bigl(\rho_r(\gamma)\bigr).
\]
\end{defn}

Thus $P_\gamma$ is monic of degree $2g$ with rational coefficients, integral when $\gamma\in\End(A)$, and $\Tr(\gamma)$ is the negative of the coefficient of $t^{2g-1}$ in $P_\gamma(t)$.

\subsection{The Rosati involution}\label{ss:rosati}

Assume now that the abelian variety $A$ comes equipped with an ample line bundle $L$, and denote by $\lambda:A\to A^{\vee}$, $\lambda(x)=t_x^*L\otimes L^{-1}$, the induced polarization.  The first Chern class of $L$ is a nondegenerate alternating form
\[
E=E^{L}: \Lambda_{\mathbb{Q}}\times\Lambda_{\mathbb{Q}} \longrightarrow\mathbb{Q},
\]
known as the \emph{Riemann form} of $L$.  Extending $E$ to $V=\Lambda_{\mathbb{R}}$ by $\mathbb{R}$-bilinearity, the Hermitian form $H(x,y)=E(ix,y)+iE(x,y)$ on $V$ is positive definite \cite[\S2.1]{BirLan}, and we write
\[
S(x,y):=E(ix,y)=\operatorname{Re}H(x,y),
\]
a positive definite symmetric bilinear form on the real vector space $V$.

The \emph{Rosati involution} with respect to $L$ is defined as
\[
\dagger:\End_{\mathbb{Q}}(A)\to\End_{\mathbb{Q}}(A), \qquad \gamma\mapsto\gamma^{\dagger}:=\lambda^{-1}\gamma^{\vee}\lambda.
\]
It is the adjoint involution with respect to $E$:
\begin{equation}\label{eq:adjoint}
E(\gamma x,y)=E(x,\gamma^{\dagger}y) \qquad\text{for all }x,y\in\Lambda_{\mathbb{Q}},
\end{equation}
where $\End_\QQ(A)$ acts on $\Lambda_\QQ$ through the rational representation. Since $\rho_a(\gamma)$ and $\rho_a(\gamma^{\dagger})$ are $\mathbb{C}$-linear, $\dagger$ is equally the adjoint with respect to $S$:
\begin{equation}\label{eq:adjointS}
S(\gamma x,y)=S(x,\gamma^{\dagger}y) \qquad\text{for all }x,y\in V .
\end{equation}
An endomorphism $\gamma\in\End_{\mathbb{Q}}(A)$ is called \emph{symmetric} if $\gamma=\gamma^{\dagger}$.

Two abelian subvarieties $B$ and $C$ of $A$ are said to be \emph{orthogonal} if $\iota_C^{\vee}\circ\lambda\circ\iota_B=0$, where $\iota_*:*\hookrightarrow A$ denotes the inclusion map; equivalently, $E(\Lambda_{B,\mathbb{Q}},\Lambda_{C,\mathbb{Q}})=0$. If, in addition, $B$ and $C$ have complementary dimensions, then $B$ and $C$ are \emph{orthogonal complements}, and we write $C=B^{\perp}$; on rational homology this reads
\begin{equation}\label{eq:perp}
\Lambda_{B^{\perp},\mathbb{Q}}=\Lambda_{B,\mathbb{Q}}^{\perp} :=\{x\in\Lambda_{\mathbb{Q}}: E(x,\Lambda_{B,\mathbb{Q}})=0\}.
\end{equation}
Since $E$ restricts nondegenerately to $\Lambda_{B,\mathbb{Q}}$, we have $\Lambda_{\mathbb{Q}}=\Lambda_{B,\mathbb{Q}}\oplus \Lambda_{B,\mathbb{Q}}^{\perp}$. Every abelian subvariety $B\subseteq A$ has a unique orthogonal complement $B^{\perp}=\ker^0(\iota_B^\vee\circ \lambda)$, which on rational homology is characterized by \Cref{eq:perp}.

\begin{lemma}\label{lem:orthogonality}
If $\gamma\in\End_{\mathbb{Q}}(A)$, then we have
\begin{enumerate}
    \item $(\image\gamma^{\dagger})^{\perp}=\ker^{0}\gamma$, and
    \item $\ker^{0}\gamma=\ker^{0}(\gamma^{\dagger}\circ\gamma)$.
\end{enumerate}
In particular, if $\gamma$ is symmetric, then $\image\gamma$ and $\ker^{0}\gamma$ are orthogonal complements.
\end{lemma}

\begin{proof}
Write $T:=\rho_r(\gamma)$, using the same notation for its real extension. By \Cref{eq:adjoint} and \Cref{eq:adjointS}, $\rho_r(\gamma^{\dagger})$ is its adjoint for both $E$ and $S$, and we denote it by $T^{*}$.

For (1), let $x\in\Lambda_{\mathbb{Q}}$. Then
\[
x\in\bigl(T^{*}\Lambda_{\mathbb{Q}}\bigr)^{\perp} \iff E(x,T^{*}y)=0, \ \forall y \iff E(Tx,y)=0,\ \forall y \iff Tx=0,
\]
the last equivalence by nondegeneracy of $E$.  Hence $\bigl(\Lambda_{\image\gamma^{\dagger},\mathbb{Q}}\bigr)^{\perp} =\ker T$. The assertion follows by \Cref{lem:dictionary}(1) and \eqref{eq:perp}.

For (2), the inclusion $\ker T\subseteq\ker T^{*}T$ is clear.  Conversely, by \eqref{eq:adjointS} we have $S(x,T^{*}Tx)=S(Tx,Tx)$ for every $x\in V$; so if $T^{*}Tx=0$ then $S(Tx,Tx)=0$, whence $Tx=0$ because $S$ is positive definite.  Thus $\ker T=\ker T^{*}T$, and we conclude again by Lemma~\ref{lem:dictionary}(1).

The final assertion follows from (1) applied to $\gamma=\gamma^{\dagger}$.
\end{proof}

\begin{comment}
For an abelian subvariety $B$, there is the following commutative diagram:
\[
\begin{tikzcd}
    A \arrow{r}{\lambda} & A^\dual \arrow{d}{\iota_B^\dual}\\
    B \arrow[hook]{u}{\iota_B} \arrow{r}{\lambda_B} & B^\dual 
\end{tikzcd}
\]
where $\lambda_B$ is the isogeny restricted from $\lambda$. The $\QQ$-endomorphism
\[
\epsilon_B:=\iota_B\circ\lambda_B^{-1}\circ\iota_B^\dual\circ\lambda
\]
is a \emph{symmetric idempotent} in $\End_\QQ(A)$ such that $\image(\epsilon_B) = B$. This construction gives rise to a bijective correspondence between abelian subvarieties of $A$ and symmetric idempotents of $\End_\QQ(A)$; see \cite[Prop.~2.1]{BirLan2} for more details. The exponent of $B$, written as $e(B)$, is defined as the exponent of the finite group $\ker\lambda_B$.

A $\QQ$-endomorphism $\gamma\in \End_\QQ(A)$ induces an action on the $l$-adic rational Tate module $V_l (A) =T_l (A)\otimes \QQ$, which gives rise to a representation
\[
\rho : \End_\QQ(A) \rightarrow  \End_{\QQ_l}(V_l(A)).
\]
\end{comment}

\subsection{Degree of the kernel}
Let $(A,L)$ be the polarized abelian variety as above. The pairing induced by the Rosati involution
\[
(\cdot,\cdot):\End_{\mathbb{Q}}(A)\times \End_{\mathbb{Q}}(A)\longrightarrow\mathbb{R}, \qquad (\gamma_1,\gamma_2)\mapsto\Tr(\gamma_1^{\dagger}\gamma_2),
\]
is a positive definite symmetric bilinear form, called the \emph{Weil metric}.  Indeed, choose a basis $e_1,\dots,e_{2g}$ of the real vector space $V$ orthonormal for $S$; then for $T=\rho_r(\gamma)$ we have
\begin{equation}\label{eq:posdef}
    \Tr(\gamma^{\dagger}\gamma) =\sum_{i=1}^{2g}S\bigl(T^{*}Te_i,e_i\bigr) =\sum_{i=1}^{2g}S\bigl(Te_i,Te_i\bigr)\ge0,
\end{equation}
with equality if and only if $T=0$, that is, if and only if $\gamma=0$.  For $\gamma\in\End_{\mathbb{Q}}(A)$ we write $\lVert\gamma\rVert:=\sqrt{\Tr(\gamma^{\dagger}\gamma)}$ for the induced norm. Notice that
\[
\|[n]\|^2=2gn^2, \qquad \|\gamma\|_{L^{\otimes m}}=\|\gamma\|_L\quad(m\geq1).
\]
 
For a closed subscheme $X$ of $A$, we write 
\[
\deg (X)=\deg_L(X):=(c_1(L)^{\dim X}\cdot [X])
\]
for the projective degree of $X$ with respect to $L$. For an abelian subvariety $B$ of $A$, we write $\covol(\Lambda_B)$ for the covolume of $\Lambda_B$ inside $V_B$ with respect to $S$. The next lemma is well-known and useful to us:

\begin{lemma}\label{lem:covolumeDegree}
For every abelian subvariety $B\subseteq A$, we have $\deg B=(\dim B)!\covol(\Lambda_B)$.
\end{lemma}
\begin{proof}
    See \cite[\S3.6]{BirLan}.
\end{proof}

The next lemma justifies the use of orthogonality in studying degrees of abelian subvarieties.

\begin{lemma}\label{lem:polarizedMorphism}
Let $B$ be an abelian subvariety of $A$ of dimension $g'$, and let $B^{\perp}$ be its orthogonal complement.  Consider the isogeny $f:B\times B^{\perp}\to A$ given by $(x,y)\mapsto x+y$.  Then
\[
f^{*}L \cong p_1^{*}L_B\otimes p_2^{*}L_{B^{\perp}},
\]
where $p_1,p_2$ are the projections and $L_B,L_{B^{\perp}}$ are the restrictions of $L$ on $B,B^{\perp}$, respectively.  Consequently, we have
\[
\deg(A)\cdot\deg(B\cap B^{\perp})=\binom{g}{g'}\deg(B)\cdot\deg(B^{\perp}).
\]
\end{lemma}

\begin{proof}
    Let $\calP_A$ denote the Poincar\'e bundle on $A \times A^\dual$. The line bundle $\calL:=\mathrm{add}^*L \otimes p_1^*L^{-1} \otimes p_2^*L^{-1}$ on $A\times A$, where $\mathrm{add}: A \times A \rightarrow A$ is the addition map defined by $(x,y)\mapsto x+y$, is a divisorial correspondence. By regarding $\calL$ as a family of line bundles parametrized by the second component and noticing that $\calL|_{A\times \{ x \}}$ is nothing else than $t_x^*L \otimes L^{-1}$, the universal property of $\calP_A$ induces the unique map $(\id_A, \lambda): A\times A\rightarrow A\times A^\dual$ satisfying $\calL \cong (\id_A, \lambda)^*\calP_A.$
    
    Let us now consider the embedding $(\iota_1,\iota_2): B \times B^\perp \hookrightarrow A \times A$; we first prove that $(\iota_1, \iota_2)^*\calL \cong \calO$. To show this, let us first note that, on $B\times A^\dual$, we have
    \[
    (\iota_1,\id_{A^\dual})^*\calP_A \cong (\id_B, \iota_1^\dual)^* \calP_B
    \]
    by the universal properties of $\calP_A,\calP_B$. Thus, using that $\iota_1^\dual\circ \lambda\circ\iota_2 = 0$, we get
    \[
    \begin{split}
    (\iota_1, \iota_2)^* \calL &= (\iota_1, \lambda\circ\iota_2)^* \calP_A \\
    & = (\id_B, \lambda\circ\iota_2)^* (\iota_1,\id_{A^\dual})^* \calP_A\\
    & = (\id_B, \lambda\circ\iota_2)^* (\id_B, \iota_1^\dual)^* \calP_B\\
    & = (\id_B, \iota_1^\dual \circ \lambda\circ \iota_2) ^*\calP_B = (\id_B, 0)^*\calP_B \cong \calO,
    \end{split} 
    \]
    proving the claim. Using this, we have
    \[
    f^*L = (\mathrm{add}\circ (\iota_1,\iota_2))^* L = (\iota_1,\iota_2)^*(\calL\otimes p_1^*L\otimes p_2^*L) \cong p_1^*L_B \otimes p_2^*L_{B^\perp},
    \]
    giving the first part. 

    For the second statement, the projection formula and the fact that $f$ is finite of degree $\deg(B\cap B^{\perp})$ give
    \[
    \deg(B\cap B^{\perp})\cdot\deg(A)=\bigl((f^{*}L)^{g}\bigr) =\Bigl(\bigl(p_1^{*}L_B\otimes p_2^{*}L_{B^{\perp}}\bigr)^{g}\Bigr) =\binom{g}{g'}\bigl(L_B^{g'}\bigr)\bigl(L_{B^{\perp}}^{g-g'}\bigr).
    \]
    All terms except $k=g'$ vanish: if $k>g'$, we have $(p_1^{*}L_B)^{k}=0$, while if $k<g'$, the power on $B^\perp$ exceeds $g-g'$.
\end{proof}

We now bound the degree of the full scheme-theoretic kernel of an endomorphism in terms of its Weil norm.

\begin{prop}\label{prop:degreeBound}
Let $\gamma\in\End(A)$ be nonzero and set $d=\dim(\image\gamma)$.  Then
\begin{equation}\label{eq:degkeridentity}
    \deg(\ker\gamma)\cdot\deg(\image\gamma) =\binom{g}{d}^{-1}\deg(A)\prod_{i=1}^{2d}\sigma_i,
\end{equation}
where $\sigma_1,\dots,\sigma_{2d}$ are the nonzero singular values of $\rho_r(\gamma)$ with respect to $S$.  Consequently, we have
\[
\deg(\ker\gamma)\le \frac{(g-d)!}{g!}\cdot\frac{\deg(A)}{(2d)^{d}}\lVert\gamma\rVert^{2d} \le\frac{\deg(A)}{2g}\max\{1,\lVert\gamma\rVert\}^{2g}.
\]
\end{prop}

\begin{proof}
Let $\alpha:=\gamma^{\dagger}\circ\gamma$, which is symmetric, let $B:=\image\alpha$ and $C:=\image\gamma$, and let $B^{\perp}$ be the orthogonal complement of $B$, which equals $\ker^{0}\alpha=\ker^{0}\gamma$ by \Cref{lem:orthogonality}. Thus $\dim B=\dim C=d$, and $\gamma$ restricts to an isogeny $\gamma_B:B\to C$.

Write $T:=\rho_r(\gamma)$. By \eqref{eq:adjointS} the operator $T^{*}T$ is self-adjoint and positive semidefinite for $S$, so it is diagonalizable with nonnegative real eigenvalues; its kernel is $\ker T$ by \Cref{lem:orthogonality}, so exactly $2d$ of these eigenvalues are positive. Denote by $\sigma_1,\dots,\sigma_{2d}>0$ their square roots, that is, the nonzero singular values of $T$ for $S$. Then
\begin{equation}\label{eq:sumsigma}
    \sum_{i=1}^{2d}\sigma_i^{2}=\Tr(T^{*}T) =\Tr(\gamma^{\dagger}\gamma)=\lVert\gamma\rVert^{2}.
\end{equation}
Note that $\Lambda_{B,\mathbb{Q}}=\image(T^{*}T)$ is the $S$-orthogonal complement of $\ker T$, since $T^{*}T$ is $S$-self-adjoint.
 
%\emph{Step 1: the degree of $\gamma_B$.}  
Since $T\Lambda\subseteq\Lambda$ and $\Lambda_C=\Lambda\cap\Lambda_{C,\mathbb{Q}}$, we have $T\Lambda_B\subseteq\Lambda_C$, and $\deg\gamma_B=[\Lambda_C:T\Lambda_B]$.  The restriction of $T$ to $\Lambda_{B,\mathbb{Q}}$ has singular values $\sigma_1,\dots,\sigma_{2d}$ for $S$, hence multiplies $2d$-dimensional $S$-volume by $\prod_{i}\sigma_i$.  Comparing covolumes inside $V_C$ and applying \Cref{lem:covolumeDegree}, whose factorials cancel because $\dim B=\dim C=d$, we obtain
\begin{equation}\label{eq:deggammaB}
\deg\gamma_B =\frac{\covol(T\Lambda_B)}{\covol(\Lambda_C)} =\Bigl(\prod_{i=1}^{2d}\sigma_i\Bigr)\frac{\deg(B)}{\deg(C)} .
\end{equation}
 
%\emph{Step 2: the degree of $\ker\gamma$.}  
Consider the isogeny $f:B\times B^{\perp}\to A$, $(x,y)\mapsto x+y$, of degree
$\deg(B\cap B^{\perp})$. Since $\gamma$ vanishes on $B^{\perp}$ we have $\ker(\gamma\circ f)=\ker(\gamma_B)\times B^{\perp}$, which has $\deg\gamma_B$ connected components; as $f$ is surjective with kernel of order $\deg(B\cap B^{\perp})$, the number $N$ of connected components of $\ker\gamma$ is $N=\deg\gamma_B/\deg(B\cap B^{\perp})$. Each component is a translate of $\ker^{0}\gamma=B^{\perp}$, and has the same degree. So \Cref{lem:polarizedMorphism} gives
\begin{equation}\label{eq:degkergamma}
    \deg(\ker\gamma)=N\cdot\deg(B^{\perp}) =\deg\gamma_B\cdot\frac{\deg(B^{\perp})}{\deg(B\cap B^{\perp})} =\binom{g}{d}^{-1}\deg\gamma_B\cdot\frac{\deg(A)}{\deg(B)}.
\end{equation}
Substituting \eqref{eq:deggammaB} into \eqref{eq:degkergamma}, the factor $\deg(B)$ cancels and we obtain \eqref{eq:degkeridentity}.
 
%\emph{Step 3: the bounds.}  
By \eqref{eq:sumsigma} and the AM--GM inequality,
\[
\prod_{i=1}^{2d}\sigma_i \le \Bigl(\frac{1}{2d}\sum_{i=1}^{2d}\sigma_i^{2}\Bigr)^{d} = \frac{\lVert\gamma\rVert^{2d}}{(2d)^{d}} ,
\]
while $\deg(\image\gamma)=d!\,\chi(L_C)\ge d!$. Since $\binom{g}{d}^{-1}(d!)^{-1}=(g-d)!/g!$, the identity \eqref{eq:degkeridentity} gives the first inequality.  For the second, set $c(d):=(g-d)!/\bigl(g!\,(2d)^{d}\bigr)$ and note that
\[
\frac{c(d+1)}{c(d)}=\frac{1}{g-d} \cdot\frac{(2d)^{d}}{(2d+2)^{d+1}}<1 \qquad (1 \le d < g).
\]
Hence $c$ is decreasing and $c(d)\le c(1)=1/(2g)$ for every $1\le d\le g$. Finally, since $\lVert\gamma\rVert^{2d} \le \max\{1,\lVert\gamma\rVert\}^{2g}$, the second inequality follows.
\end{proof}

\begin{remark}\label{rmk:sharp}
(1) Equality in the first inequality of \Cref{prop:degreeBound} holds if and only if all nonzero singular values of $\rho_r(\gamma)$ coincide and the restriction $L|_{\operatorname{im}\gamma}$ induces a principal polarization on $\image\gamma$; for example, it holds for $A=E_1\times E_2$ a product of principally polarized elliptic curves and $\gamma$ the projection onto $E_1$.

(2) Even for a fixed polarized abelian variety, \(\deg(\ker\gamma)\) need not tend to infinity with \(\|\gamma\|\). In particular, there is no lower bound depending only on \(\|\gamma\|\) that tends to infinity with the norm. For example, if $A=E\times E$ where $E$ is an elliptic curve and $n$ is a positive integer, take 
\[
\gamma_n = \begin{bmatrix}
1 & n \\
n & n^2+1
\end{bmatrix}
\in \End(E\times E).
\]
Then one checks that $\ker \gamma_n$ is always trivial while $\lVert\gamma_n\rVert$ grows with $n$.
\end{remark}

\section{Abelian schemes}
From now on, we focus on abelian varieties in families. This section recalls the necessary facts from the bi-algebraic geometry of the universal abelian scheme, most notably Gao's \emph{mixed Ax--Schanuel theorem} \cite{Gao_mixed}; see the survey by Klingler, Ullmo and Yafaev \cite{KUY_survey} for an elegant introduction to bi-algebraic geometry. We define the \emph{degeneracy locus} following Gao, which plays a crucial role in the proof of our main theorem through the height bound established by the third-named author in \cite{Ge_IntersectI}. Then we discuss the rational representation in the family setting. Finally, we prove a uniform bound for integral representatives of Betti cosets, which will be used in the proof of \Cref{thm:main}.

\subsection{Bi-algebraic geometry and Ax--Schanuel}\label{subsec:Bialg}
Fix $g\geq 1$. Let $\bA_g$ be the complex fine moduli space of principally polarized abelian varieties with a symplectic level-$3$-structure. It is known that $\bA_g$ is a smooth irreducible quasi-projective variety defined over the number field $\QQ[\zeta_3]$. By our convention of number fields $\QQ(\zeta_3)\subseteq\CC$, there is a variety $\AAA_g$ over $\QQ(\zeta_3)$ with $\bA_g\cong \AAA_g\times_{\QQ(\zeta_3)}\CC.$ Let $\pi_g:\gothA_g\rightarrow \bA_g$ be the universal abelian scheme.

The following two transcendental uniformization maps will be relevant.
\begin{enumerate}
    \item The analytification $\bA_g^\textnormal{an}$ admits a uniformization $u_g:\gothH_g \to \bA_g^\mathrm{an}$ from the Siegel upper half space $\gothH_g:=\{Z \in \Mat_g(\CC): Z=Z^t, \ \textnormal{Im}(Z)>0\}$.
    
    \item The analytification $\gothA_g^\mathrm{an}$ admits a uniformization $u_{g,a}:\CC^g\times \gothH_g\rightarrow \gothA_g^\mathrm{an}$.
\end{enumerate}

We endow the universal covers with an algebraic structure as follows: an \emph{irreducible algebraic subvariety} of $\gothH_g$ (resp.\ of $\CC^g\times \gothH_g$) is an irreducible analytic component of $U\cap\gothH_g$ (resp.\ of $\calU\cap(\CC^g\times \gothH_g)$), where $U$ (resp.\ $\calU$) is an algebraic subvariety of the affine spac $\Sym_g(\CC)$ (resp.\ $\CC^g\times \Sym_g(\CC)$); here we use the inclusion of $\gothH_g$ as an open subset of $\Sym_g(\CC)$ in the complex-analytic topology. Bi-algebraic geometry \cite{KUY_survey} studies the interplay of these two algebraic structures, on source and target, across such a transcendental map.

\begin{defn}
An irreducible algebraic subvariety $Y \subseteq \bA_g$ (resp. $\calY \subseteq \gothA_g$) is \emph{bi-algebraic} if one, equivalently every, complex-analytic irreducible component of $u_g^{-1}(Y^\mathrm{an})$ (resp. $u_{g,a}^{-1}(\calY^{\mathrm{an}})$) is algebraic. A \emph{bi-algebraic subset} is a finite union of bi-algebraic subvarieties.
\end{defn}

The class of bi-algebraic subsets is closed under intersection, so for any irreducible subset $Z$ there is a smallest bi-algebraic subvariety containing it, obtained as the intersection of all bi-algebraic subvarieties containing $Z$. We call it the \emph{bi-algebraic closure} of $Z$ and denote it $Z^{\mathrm{biZar}}$.

Ullmo and Yafaev~\cite{UY_special} classify the irreducible bi-algebraic subvarieties of $\bA_g$ as the \emph{weakly special} subvarieties (we omit the definition, as it will not be used directly). Gao~\cite{Gao_APZ} shows further that an irreducible bi-algebraic subvariety $\calY$ of $\gothA_g$ has the form
\[
\calY = \calB+ \sigma_c
\]
where the base $S=\pi_g(\calY)$ is bi-algebraic in $\bA_g$, $\calB$ is an abelian subscheme of $\gothA_g|_{S}$, and $\sigma_c$ is a constant multi-section of $\gothA_g|_{S}$ in the sense of \cite[Def.~2.4]{Ge_IntersectI}.

The mixed Ax--Schanuel theorem, in this case established by Gao~\cite{Gao_mixed}, states informally that ``atypical'' intersections between algebraic subvarieties of $(\CC^g\times \gothH_g)\times \gothA_g^\mathrm{an}$ and the graph of $u_{g,a}$ are controlled by bi-algebraic subvarieties. Rather than state the general theorem, we record the consequence we shall use.

\begin{prop}\label{prop:AxScCor}
    Let $\calU\subseteq \CC^g\times \gothH_g$ be a closed irreducible analytic subset. Then,
    \[
    \rdim \calU^\mathrm{Zar} - \dim \calU \geq \rdim u_{g,a}(\calU)^\mathrm{biZar} - \dim u_{g,a}(\calU)^\mathrm{Zar},
    \]
    where $\calU^{\mathrm{Zar}}$ is the Zariski closure of $\calU$ in the algebraic structure of the cover, $\rdim\calU^{\mathrm{Zar}}$ is taken over $\gothH_g$, and $\rdim u_{g,a}(\calU)^\mathrm{biZar}$ over $\bA_g$.
\end{prop}

\begin{proof}
    The statement follows from the weak Ax--Schanuel theorem~\cite[Theorem~3.5]{Gao_mixed}, on noticing that $\pi_g\bigl(u_{g,a}(\calU)^\mathrm{biZar}\bigr)$ contains $u_g\bigl(p(\calU^\mathrm{Zar})\bigr)$, where $p:\CC^g\times \gothH_g\rightarrow \gothH_g$ is the projection.
\end{proof}

\subsection{Weakly special subvarieties and degeneracy loci}\label{subsec:WeaklySpecial}
Let $\pi:\calA\rightarrow S$ be an abelian scheme of relative dimension $g$ over a normal quasi-projective complex variety $S$. For convenience of presentation, we assume throughout this section that $\calA/S$ is equipped with a principal polarization and a symplectic level-$3$-structure. There is then a Cartesian diagram
\begin{equation}\label{eqn:modularMaps}
\begin{tikzcd}
    \calA \arrow{r}{\iota_\calA} \arrow{d}{\pi} \arrow[dr, phantom, "\square"] & \gothA_g \arrow{d}{\pi_g} \\
    S \arrow{r}{\iota_S} & \bA_g
\end{tikzcd}
\end{equation}
where $\iota_S: S\rightarrow \bA_g$ is the period map and $\iota_\calA: \calA\rightarrow \gothA_g$ is the mixed period map. The level-$3$-structure further ensures that $\End(\calA_\eta)=\End(\calA_{\bar \eta})$, where $\calA_\eta$ (resp.\ $\calA_{\bar\eta}$) is the generic (resp.\ geometric generic) fiber. The reader can find the basics of abelian schemes in \cite[\S2]{Ge_IntersectI}.

\begin{defn}\label{defn:ws}
Let $\calW\subseteq\calA$ be an irreducible subvariety with $T=\pi(\calW)$, and let $\eta$ be the generic point of $T$. We call $\calW$ \emph{weakly special} if
\[
\calW=(B+\xi)^{\mathrm{Zar}}
\]
where $B\subseteq\calA_\eta$ is an abelian subvariety and $\xi\in\calA_\eta$ is a constant generic point. Here $+$ denotes the addition on $\calA$.
\end{defn}

    For an irreducible subvariety $\calY$ of $\calA$, the intersection of all weakly special subvarieties containing $\calY$ is again weakly special, and is the smallest such; we call it the \emph{weakly special closure} of $\calY$ and write $\calY^\mathrm{ws}$. We define the \emph{vertical defect} of $\calY$ as
\[
\delta_v(\calY):=\rdim\calY^\mathrm{ws} -\dim \calY .
\]
    Since \Cref{defn:ws} imposes no condition on the base $T$, the weakly special closure has the same image in $S$ as $\calY$, namely $\pi(\calY^{\mathrm{ws}})=\pi(\calY)$; hence $\rdim\calY^{\mathrm{ws}}=\dim\calY^{\mathrm{ws}}-\dim\pi(\calY)$ and
\[
\delta_v(\calY)=\dim\calY^\mathrm{ws} -\dim\pi(\calY)-\dim \calY.
\]
For instance, if $\pi(\calY)$ is a point, $\delta_v$ reduces to the
``geodesic defect'' of \cite{HP_ominimality}; if $\calY$ is a torsion section of $\calA/S$, then $\calY^{\mathrm{ws}}=\calY$ and $\delta_v(\calY)=-\dim S$.

\begin{remark}\label{rmk:ws-vs-mixed-shimura}
Assume $S\subseteq\bA_g$, so that $\calA=\gothA_g|_S$. Our \Cref{defn:ws} differs from the notion of weakly special subvariety standard in the mixed Shimura setting (equivalently, the bi-algebraic subvarieties of $\gothA_g$ by \cite{Gao_APZ}): there one requires in addition that the base $T=\pi(\calW)$ be bi-algebraic in $\bA_g$. Consequently, for an irreducible $\calY\subseteq\calA$ the two closures need not agree. In our terminology, the bi-algebraic closure $\calY^{\biZar}$ is the ``weakly special closure'' of the mixed Shimura literature, while the weakly special closure $\calY^{\ws}$ in our sense can in general be strictly contained in $\calY^{\biZar}$. 
 
The two closures differ only in the base direction: both have the form
$\calB+\sigma_c$ with the \emph{same} abelian subscheme $\calB$,
and constant multi-section $\sigma_t$, and
$\calY^{\ws}$ is an irreducible component of $\mathcal Y^{\mathrm{biZar}}\cap
\pi_g^{-1}\bigl(\pi_g(\mathcal Y)\bigr)$. Moreover, $\pi_g(\mathcal Y^{\mathrm{biZar}}) = \pi_g(\mathcal Y)^{\mathrm{biZar}}$, and
\[
\rdim\calY^{\ws}=\rdim\calY^{\biZar}=\rdim\calB
\]
see \cite[Cor.~5.4]{Gao_Betti}. Therefore the vertical defect can be computed using either closure.
\end{remark}

\begin{defn}\label{defn:t-deg}
    Let $\calX$ be a subvariety of $\calA$. For $t\in\ZZ$, the \emph{$t$-degeneracy locus} of $\calX$ is
    \[
    \calX^{\mathrm{deg}}(t):=\bigcup\calY,
    \]
    the union over positive-dimensional irreducible subvarieties
    $\calY\subseteq \calX$ with $\delta_v(\calY)<t$. We say $\calX$ is \emph{$t$-nondegenerate} if $\calX^\mathrm{deg}(t)$ is not Zariski dense in $\calX$.
\end{defn}

The notion of degeneracy locus was introduced by Gao~\cite{Gao_Betti}, generalizing the (b-)anomalous locus studied by R\'emond~\cite[\S3]{Remond_IntersectionIII} when $S$ is a point and, for tori, by Bombieri, Masser and Zannier~\cite{BMZ_Anomalous, BMZ_plane08}. To fix ideas, suppose that $\calX$ is closed and contained in a single fiber of $\pi$, so that $\pi(\calX)$ is a point. Then every irreducible $\calY\subseteq\calX$ has $\delta_v(\calY)=\dim\calY^\mathrm{ws}-\dim\calY\geq 0$; in particular $\calX^{\deg}(t)=\emptyset$ for $t\leq 0$, and $\calX^{\deg}(1)$ is the union of the positive-dimensional cosets contained in $\calX$, that is, the classical Kawamata--Ueno locus~\cite{Ueno_Class}.

Returning to the general case, we record two properties, both immediate from the definition. First, $\calX^{\deg}(t)\subseteq\calX^{\deg}(t+1)$ for every $t\in\ZZ$, so the loci form an increasing filtration. Second, since $\rdim\calY^{\ws}\le g$ for every $\calY$, one has $\delta_v(\calX)\le g-\dim\calX$. Hence, if $\calX$ is positive-dimensional, then $\calX^{\deg}(t)=\calX$ for all $t>g-\dim\calX$; thus the filtration exhausts $\calX$. Extending the Zariski closedness of the Kawamata--Ueno locus, Gao showed that $\calX^{\deg}(t)$ is Zariski closed for every $t$; see \cite[\S7]{Gao_Betti} when the period map is an inclusion, and \cite[Thm.~5.6]{Ge_IntersectI} for the general case.

\begin{remark}\label{rmk:descent}
    If the abelian scheme $\mathcal A\to S$ and the subvariety $\mathcal X\subseteq\mathcal A$ are defined over a subfield $K\subseteq\CC$, then $\calX^\mathrm{deg}(t)$ is defined over $K$; cf.\ \cite[Prop.~8.1]{Ge_IntersectI}. We therefore speak freely of degeneracy loci for subvarieties of abelian schemes over $K$.
\end{remark}

The following algebraic criterion of Gao is our main tool for testing nondegeneracy.

\begin{theorem}[Gao]\label{thm:Gao_criterion}
Let $S$ be a normal quasi-projective complex variety, and let $\pi:\calA\rightarrow S$ be an abelian scheme of relative dimension $g$ with a symplectic level-$3$-structure. Suppose $\calX\subseteq\calA$ is irreducible, dominant over $S$, and not contained in any strict group subscheme. For any $t\in\ZZ$, the following are equivalent:
\begin{enumerate}
    \item $\calX=\calX^\mathrm{deg}(t)$;
    
    \item there is an abelian subscheme $\calB\subseteq\calA$ of relative dimension $g'$ such that the composition
    \[
    \varphi_{/\calB}:\calA\longrightarrow\calA/\calB\longrightarrow\gothA'
    \]
    of the fibrewise quotient $\calA\to\calA/\calB$ with a mixed period map $\calA/\calB\to\gothA'$ of the quotient abelian scheme satisfies
    \[
    \dim\varphi_{/\calB}(\calX)<\dim\calX-\max\{0,\,g'-t\}.
    \]
\end{enumerate}
Here $\gothA'$ denotes a universal family for $\calA/\calB$; the image dimension $\dim\varphi_{/\calB}(\calX)$ is intrinsic to $\calA/\calB$ and independent of the choice of $\gothA'$ and of the period map into it.
\end{theorem}
\begin{proof}
    See \cite[Thm.~2.4]{Gao_Survey} for the case $S\subseteq\bA_g$, and \cite[Thm.~5.7]{Ge_IntersectI} for the generalization.
\end{proof}

\subsection{Relative representations and Betti coordinates}\label{subsec:rrr}
We extend the notion of analytic and rational representations presented in \Cref{subsec:rationalRepresentation} to the relative setting. Let $\pi:\calA \to S$ be as above. Denote by $\calV:=\Lie(\calA/S)$ the Lie algebra, a holomorphic vector bundle of rank $g$ on $S$, and by $\Lambda_S:=(R^1\pi_*\ZZ)^\dual$ the local system of periods, a locally constant sheaf of free $\ZZ$-modules of rank $2g$ embedded fibrewise as a lattice in $\calV$.  The fibre-wise analytic representation induces an injective homomorphism
\[
\rho_a:\End(\calA/S) \hookrightarrow \End(\calV)
\]
into the endomorphisms of the vector bundle, which we call the \emph{relative analytic representation} of $\End(\calA/S)$.  Meanwhile, an endomorphism $\phi\in \End(\calA/S)$ induces $\phi_*:\Lambda_S \to \Lambda_S$ that uniquely determines $\phi$, thus defining an injective homomorphism of algebras
\[
\rho_r:\End(\calA/S) \hookrightarrow \End(\Lambda_S)
\]
which we call the \emph{relative rational representation} of $\End(\calA/S)$.

\begin{remark}\label{rmk:rrrMatrix}
Let $U$ be a nonempty connected complex manifold with a map $U\rightarrow S^\mathrm{an}$, over which both pullbacks $\calV_U$ and $\Lambda_U$ are trivial. Fix a holomorphic trivialization $\calV_U\simeq \CC^g\times U$ and a flat trivialization $\Lambda_U\simeq \ZZ^{2g}\times U$. The resulting fibrewise period matrices give rise to a holomorphic \emph{period map}
\[
\Pi:U \to \Mat_{g\times 2g}(\CC),
\]
whose columns are the coordinates of a $\ZZ$-basis of the period lattice of $\calA_s$. Let $\phi\in \End(\calA/S)$. The trivialization on $\calV_U$ represents $\rho_a(\phi)$ by a holomorphic function $M_a(\phi,\cdot):U \to \Mat_g(\CC)$, while the flat trivialization on $\Lambda_U$ represents $\rho_r(\phi)$ by a matrix $M_r(\phi) \in \Mat_{2g}(\ZZ)$ that is \emph{constant} in $s$: a morphism of local systems is flat, so in a flat trivialization over the connected set $U$ it is a single integer matrix, determined by its value $M_r(\phi)$ under the trivialization of $\Lambda_{s_0}\simeq\ZZ^{2g}$ at any base point $s_0\in U$ and propagated to $\Lambda_U$ by continuity. In particular, the relative rational representation induces an injective homomorphism
\[
M_r=M_{r,U}=M_{r,s_0}:\End(\calA/S) \hookrightarrow \Mat_{2g}(\ZZ).
\]
Fibrewise, we have the compatibility
\begin{equation}\label{eqn:relativeCompatibility}
    M_a(\phi,s)\,\Pi(s) = \Pi(s)\,M_r(\phi) \qquad \textnormal{for any } s \in U,
\end{equation}
which is the relative form of \eqref{eqn:compatibility}.
\end{remark}

Consider the universal abelian scheme $\pi_g:\gothA_g\rightarrow \bA_g$ and its pullback, denoted as $\mathbb{T}_g$, to the Siegel upper half space $\gothH_g$:
\[
\begin{tikzcd}
    \mathbb{T}_g \arrow[r] \arrow[d, "\tilde\pi_g"]\arrow[dr, phantom, "\square"] & \gothA_g^{\mathrm{an}} \arrow[d]\\
    \gothH_g \arrow[r] & \bA_g^\mathrm{an}
\end{tikzcd}.
\]
The universal cover $\CC^g\times \gothH_g$ of $\gothA_g^\mathrm{an}$ is identified with $\Lie(\mathbb{T}_g/\gothH_g)$. Let $\Lambda_{\gothH_g}:=(R^1\tilde{\pi}_{g*}\ZZ)^\dual$ be the local system of periods naturally embedded in $\CC^g\times \gothH_g$. Throughout we fix the trivialization of $\Lambda_{\gothH_g}$ to $\ZZ^{2g}\times \gothH_g$ given by
\begin{equation}\label{eqn:trivializationUnivLattice}
\begin{split}
   \ZZ^g\times \ZZ^g \times \gothH_g  &\xrightarrow{\sim} \Lambda_{\gothH_g}\\
    (a,b,\tau) &\mapsto (\tau a+ b, \tau),
\end{split}
\end{equation}
so that the period map of the universal family is the normalized period matrix $\Pi(\tau)=(\tau\ \ I_g)$. In particular, at any point $\tau_0\in \gothH_g$, the rational representation is a well-defined map
\[
M_{r,\tau_0}:\End(\mathbb{T}_{g,\tau_0})\rightarrow \Mat_{2g}(\ZZ)
\]
by the fixed trivialization.

Assume for now $S\subseteq\bA_g$ and let $\calA=\gothA_g|_S$ be the induced family. Fix once and for all a connected component $\widetilde S$ of $u_g^{-1}(S^{\mathrm{an}})\subseteq\gothH_g$. The inclusion $\widetilde S\hookrightarrow\gothH_g$ pulls back the fixed trivialization \eqref{eqn:trivializationUnivLattice} of $\Lambda_{\gothH_g}$ to a flat trivialization of $\Lambda_S$ over $\widetilde S$; applying \Cref{rmk:rrrMatrix} with $U=\widetilde S^{\mathrm{sm}}$, this determines
\begin{equation}\label{eqn:ratioReprMatrix}
M_r=M_{r,\widetilde S}:\End(\calA/S)\hookrightarrow\Mat_{2g}(\ZZ),
\end{equation}
a well-defined integer matrix representation, together with the compatibility \eqref{eqn:relativeCompatibility}. We use this trivialization, and the induced $M_r$, throughout.
 
The choice of component matters only up to the deck group: a different component of $u_g^{-1}(S^{\mathrm{an}})$ differs by some $\gamma\in\Sp_{2g}(\ZZ)$, under which $M_r(\phi)$ is replaced by $\gamma\,M_r(\phi)\,\gamma^{-1}$. In particular $M_r(\phi)$ is not, as an individual matrix, an invariant of $S$; only its conjugacy class is, so any quantity attached to $S$ that we extract from $M_r(\phi)$ will be a conjugacy invariant (its characteristic polynomial, trace, or the norm induced by the Rosati involution). Fixing $\widetilde S$ removes this ambiguity for the computations to follow.

\medskip
Consider the real extension of \Cref{eqn:trivializationUnivLattice}
\begin{equation}\label{eqn:realIsoUniversalCover}
\begin{split}
   \RR^g\times \RR^g \times \gothH_g  &\xrightarrow{\sim} \CC^g\times\gothH_g\\
    (a,b,\tau) &\mapsto (\tau a+b, \tau).
\end{split}
\end{equation}
This is a real-analytic isomorphism which, for each fixed $\tau$, is the $\RR$-linear isomorphism $\RR^{2g}\xrightarrow{\sim}\CC^g$ determined by the period matrix $\Pi(\tau)$.

\begin{defn}
The \emph{universal uniformized Betti map} is defined to be the map
\begin{equation}\label{eq:universalUniformizedBetti}
\widetilde{\gothb}: \CC^g\times \gothH_g \rightarrow \RR^{2g},
\end{equation}
which is the composition of the inverse of \Cref{eqn:realIsoUniversalCover} with the natural projection.  Explicitly, $\widetilde{\gothb}(z,\tau)=(a,b)$ where $z=\tau a+b$.
\end{defn}

% Since the abelian scheme $\calA/S$ is equipped with the data of \Cref{eqn:modularMaps}, the modular map $S^\mathrm{an}\to\bA_g^{\mathrm{an}}$ lifts, on the universal cover, to a map $\widetilde S\to\gothH_g$, along which $\calA/S$ is  the pullback of the universal family $\mathbb{T}_g/\gothH_g$; the level
% structure of \Cref{eqn:modularMaps} rigidifies this lift, so it is unique.
% While the local system $\Lambda_S$ cannot generally be trivialized over $S^\mathrm{an}$, its pullback $\Lambda_{\widetilde S}$ inherits the canonical trivialization $(a,b,\tau)\mapsto(\tau a+b,\tau)$ from $\Lambda_{\gothH_g}$. Because $\End(\calA/S)$ injects naturally into $\End(\calA \times_S \widetilde S / \widetilde S)$, evaluating a global endomorphism on this trivialized cover yields a well-defined relative rational representation
% \begin{equation}\label{eqn:ratioReprMatrix}
% M_r: \End(\calA/S) \hookrightarrow \Mat_{2g}(\ZZ).
% \end{equation}

\subsection{Small representatives of cosets}\label{ss:smallrep}
Consider a polarized complex abelian variety $(A,L)$ of dimension $g$, and fix a trivialization $H_1(A,\ZZ)\simeq\ZZ^{2g}$. The polarization equips $H_1(A,\RR)\simeq \RR^{2g}$ with a Euclidean metric, namely the positive definite symmetric form $S=\operatorname{Re}H$ associated with the Riemann form of $L$ as in Section~\ref{ss:rosati}; we refer to it as the \emph{polarization metric} of $(A,L)$. We denote it by $(\cdot,\cdot)$, and its associated norm by $\lVert\cdot\rVert$.  Let $\Lambda$ be a subgroup of $\ZZ^{2g}$ of rank $r\ge1$. 

In the point-counting argument of the main proof, we will encounter real cosets $z+\Lambda_\RR\subseteq\RR^{2g}$, where $\Lambda\subseteq\ZZ^{2g}$ is the period lattice of an abelian subvariety in the Betti coordinates, and we will need to replace a coset by an integral representative $\lambda\in\ZZ^{2g}$ of controlled norm.  Finding a reasonably good representative in a uniform way is the goal of this subsection.

\begin{defn}
The \emph{covering radius} $\mu(\Lambda)$ of $\Lambda$ is the smallest radius such that the closed balls of radius $\mu(\Lambda)$ centered at the points of $\Lambda$ cover $\Lambda_\RR$.
\end{defn}

We write $\covol(\Lambda)$ for the covolume of $\Lambda$ in $\Lambda_\RR$ under the polarization metric. Recall that the covolume of a period lattice of an abelian variety is proportional to the degree by \Cref{lem:covolumeDegree}. The next lemma bounds the covering radius by the covolume and the first successive minimum:

\begin{lemma}\label{lem:coveringRadius}
$\mu(\Lambda)\leq c_r\,\mu_1^{1-r}\covol(\Lambda)$, where $c_r$ is a constant depending only on $r$, and $\mu_1=\min\{\lVert\lambda\rVert:\lambda\in\Lambda\smallsetminus \{0\}\}$.
\end{lemma}

\begin{proof}
Let $\mu_1\leq\cdots\leq\mu_{r}$ be the successive minima of $\Lambda$. By definition, there are linearly independent vectors $\lambda_1,\dots,\lambda_{r}\in\Lambda$ spanning $\Lambda_\RR$ with $\lVert\lambda_i\rVert=\mu_i$.  In particular, for any $z\in\Lambda_\RR$, there are $c_1,\dots,c_{r}\in\RR$ with $|c_i|\leq 1/2$ such that $z-c_1\lambda_1-\cdots-c_{r}\lambda_{r}\in\Lambda$.  This implies that
\[
\mu(\Lambda)\leq\frac{1}{2}\sum_{i=1}^r\mu_i .
\]
By Minkowski's second theorem, $\mu_1\cdots\mu_{r}\leq 2^{r}\covol(\Lambda)/V_r$, where $V_r$ is the volume of the standard $r$-dimensional unit ball.  Using $\mu_1\cdots\mu_r\geq\mu_1^{r-1}\mu_r$, we get
\[
\mu(\Lambda)\leq\frac{r}{2}\,\mu_r \leq\frac{r\cdot 2^{r-1}}{V_r}\,\mu_1^{1-r}\covol(\Lambda).
\]
Set $c_r=r\cdot 2^{r-1}/V_r$ and we are done.
\end{proof}

The lemma above and \Cref{lem:covolumeDegree} are stated for the polarization metric, whereas the point counting takes place in the fixed coordinates $\RR^{2g}$ with the \emph{standard Euclidean metric} $|\cdot|$. Two observations make the transfer uniform for a family $\calA/S$ with $S\subseteq\bA_g$ and $\calA=\gothA_g|_S$. First, fix a connected component $\widetilde S$ of $u_g^{-1}(S^{\mathrm{an}})$ and use the induced trivialization $\Lambda_{\widetilde S}\simeq\ZZ^{2g}\times \widetilde S$ of \eqref{eqn:trivializationUnivLattice}. For an abelian subscheme $\calB\subseteq\calA$, the local system $H_1(\calB\times_S \widetilde S/\widetilde S,\ZZ)$ defines a map from $\widetilde S$ to the set of sublattices of $\ZZ^{2g}$; this map is locally constant, and the set of sublattices is discrete, so it is constant on the connected $\widetilde S$. Thus the abelian subscheme $\calB$ determines a well-defined sublattice $\Lambda\subseteq\ZZ^{2g}$, and each quantity attached to it may be evaluated at one point of $\widetilde S$, chosen once and for all. Second, since $\Lambda\subseteq\ZZ^{2g}$, its first successive minimum is bounded below in terms of the comparison constant between the two metrics at that point. This gives the following lemma, the goal of this subsection.
 
\begin{lemma}\label{lem:periodNorm}
Let $S\subseteq\bA_g$, let $\calA=\gothA_g|_S$ be the induced abelian scheme and let $\calL$ be a relatively ample line bundle. Fix a connected component $\widetilde S$ of $u_g^{-1}(S^{\mathrm{an}})$, and use the induced trivialization $\Lambda_{\widetilde S}\simeq\ZZ^{2g}\times \widetilde S$ from \eqref{eqn:trivializationUnivLattice}. There exists a constant $c>0$, depending only on $(\calA/S,\calL)$ and the choice of $\widetilde S$, with the following property.
 
Let $\calB$ be an abelian subscheme of $\calA$, and let $\Lambda\subseteq\ZZ^{2g}$ be the lattice corresponding to $\calB$. Then every $z\in\ZZ^{2g}+\Lambda_\RR$ admits $\lambda\in\ZZ^{2g}$ with $z-\lambda\in\Lambda_\RR$ and
\[
|\lambda| \leq |z|+c\deg_\calL\calB ,
\]
where $|\cdot|$ denotes the standard Euclidean norm.
\end{lemma}
 
\begin{proof}
Since the conclusion involves only $\Lambda$, $\deg_\calL\calB$ and the standard metric, none of which varies over $\widetilde S$, we may fix once and for all a point $\tau_0\in\widetilde S$ and argue there; all constants below are computed at $\tau_0$ and thus depend only on $(\calA/S,\calL)$ and $\widetilde S$.
 
Let $\lVert\cdot\rVert$ be the polarization metric of $(\calA_{\tau_0},\calL_{\tau_0})$ on $\RR^{2g}$, and let $c'\geq1$, depending only on $\tau_0$, satisfy
\[
\tfrac{1}{c'}|v|\leq \lVert v\rVert \leq c' |v| \qquad\text{for all }v\in\RR^{2g}.
\]
Set $g'=\dim(\calB/S)$; if $g'=0$ the statement holds trivially with $\lambda=z$, so assume $g'\geq1$, and note the sublattice $\Lambda$ determined by $\calB$ has rank $2g'$.
 
Every nonzero $\lambda\in\Lambda\subseteq\ZZ^{2g}$ has $|\lambda|\geq1$, hence $\lVert\lambda\rVert\geq1/c'$. Therefore $\mu_1=\mu_1(\Lambda,\lVert\cdot\rVert)\geq1/c'$, and since $1-2g'<0$, we have
\[
\mu_1^{1-2g'} \leq (c')^{2g'-1}.
\]
Combining \Cref{lem:coveringRadius} (with $r=2g'$) and \Cref{lem:covolumeDegree},
\begin{equation}\label{eq:muBound}
\mu(\Lambda)\leq c_{2g'}\,\mu_1^{1-2g'}\covol(\Lambda) \leq \frac{c_{2g'}(c')^{2g'-1}}{g'!}\deg_\calL\calB.
\end{equation}
 
Write $z\in\lambda_1+\Lambda_\RR$ with $\lambda_1\in\ZZ^{2g}$, and decompose $z=z^{\parallel}+z^{\perp}$ with $z^{\parallel}\in\Lambda_\RR$ and $z^{\perp}$ orthogonal to $\Lambda_\RR$ for the standard metric. Then $z^{\perp}=z-z^{\parallel}\in\lambda_1+ \Lambda_\RR$, and $z^{\perp}$ is the point of $\lambda_1+\Lambda_\RR$ closest to the origin, so $|z^{\perp}|=\operatorname{dist}(0,z+\Lambda_\RR)$.
 
By definition of the covering radius, applied in $\lVert\cdot\rVert$ for $\lambda_1+\Lambda\subseteq\lambda_1+\Lambda_\RR \ni z^\perp$, there is $\lambda\in\lambda_1+\Lambda\subseteq\ZZ^{2g}$ with $\lVert z^{\perp}-\lambda\rVert\leq\mu(\Lambda)$. In particular $z-\lambda\in\Lambda_\RR$, and by \eqref{eq:muBound}
\[
|z^{\perp}-\lambda| \leq c'\lVert z^{\perp}-\lambda\rVert \leq \frac{c_{2g'}(c')^{2g'}}{g'!}\deg_\calL\calB .
\]
Finally,
\[
|\lambda| \leq |z^{\perp}|+|z^{\perp}-\lambda| \leq \operatorname{dist}\bigl(0,z+\Lambda_\RR\bigr)+c\deg_\calL\calB
\]
with $c=\max_{1\leq g'\leq g}c_{2g'}(c')^{2g'}/g'!$, and $\operatorname{dist}(0,z+\Lambda_\RR)\leq|z|$.
\end{proof}

\section{Point counting and Betti coordinates}
In this section, we first recall the o-minimal point counting theorem of Habegger and Pila, together with the necessary notions from model theory, and we set up the definable family that encodes the Betti coordinates on a subvariety $\calX\subseteq\gothA_g$ under the action of an endomorphism. Then we prove the main result of this section: if the Betti coordinates of an arc of $\calX$ are confined to a semi-algebraic set after applying the rational representation of an endomorphism, then that piece lies in a degeneracy locus of $\calX$. The two ingredients are combined later in the proof of the main theorem: the counting theorem produces such an arc from an abundance of rational points, and the criterion of \Cref{subsec:algdep} then gives degeneracy.

\subsection{O-minimality}\label{subsec:o-minimality}

A key ingredient in the proof of the main result is a point-counting theorem of Habegger and Pila~\cite{HP_ominimality}, itself building on the counting theorem of Pila and Wilkie~\cite{Pila_Wilkie}. We first recall the necessary notions from model theory; for a fuller treatment see \cite{Dries98,vdDriesMiller96}.

\begin{defn}
A \emph{structure} on $\RR$ is a sequence $\calS = (\calS_n)_{n \ge 1}$ such that for any $n,m \ge 1$:
\begin{itemize}
    \item $\calS_n$ is a Boolean algebra of subsets of $\RR^n$;

    \item $\calS_n$ contains every semi-algebraic subset of $\RR^n$;

    \item $A\in\calS_n$ and $B\in\calS_m$ imply $A\times B\in\calS_{n+m}$;

    \item $A\in\calS_{n+1}$ implies $\pi(A)\in\calS_n$,  where $\pi:\RR^{n+1} \to \RR^n$ is the projection onto the first $n$ coordinates.
\end{itemize}
It is \emph{o-minimal} if, moreover, the members of $\calS_1$ are exactly the finite unions of intervals.
\end{defn}

Fix an o-minimal structure $\calS$. A set $A\subseteq\RR^n$ is \emph{definable} if $A\in\calS_n$, and a function $f:A\to\RR^m$ is \emph{definable} if its graph is a definable subset of $\RR^{n+m}$. Finally, an \emph{arc} is a non-constant continuous map $\beta:[0,1]\to\RR^{N}$; by abuse of language we also call its image an arc, using the term for either as context dictates. It is \emph{definable} if the map $\beta$ is definable.

To count rational points we use multiplicative height. For $a/b\in\QQ$ in lowest terms, $H(a/b)=\max\{|a|,|b|\}$, and for $\alpha=(\alpha_1,\dots,\alpha_n)\in\QQ^n$, $H(\alpha)=\max_i H(\alpha_i)$. Given $Z\subseteq\RR^{n+m}=\RR^n\times \RR^m$ and $T\ge1$, we record the rational points of bounded height in the first factor
\[
Z^\sim(T)=\{(y,z)\in Z:\ y\in\QQ^n,\ H(y)\le T\}.
\]
The counting theorem bounds how many such points a definable set can carry before it must contain a semi-algebraic arc after projection, in a uniform way.

To state it, let $F \subseteq \RR^l \times \RR^{n+m}$ be a family of subsets of $\RR^{n+m}$; for $\psi \in \RR^l$ its fiber is $F_{\psi}=\{x \in \RR^{n+m} : (\psi,x)\in F\}$.  Let $\pi_1:\RR^{n+m}\to \RR^n$ and $\pi_2:\RR^{n+m}\to \RR^m$ be the two projections.

\begin{prop}[Habegger--Pila, {\cite[Cor.~7.2]{HP_ominimality}}]\label{prop:HP}
    Let $F \subseteq \RR^l \times \RR^n \times \RR^m$ be a definable family and $\varepsilon>0$ a real number. There exists a constant $c_\mathrm{HP}>0$, depending only on $F$ and $\varepsilon$, with the following property. If $\psi\in \RR^l$, $T\ge 1$ and $\Sigma \subseteq F_\psi^\sim(T)$ with
    \[
    \#\pi_2(\Sigma)>c_\mathrm{HP}\,T^\varepsilon,
    \]
    then there is a definable arc $\beta:[0,1] \to F_\psi$ such that:
    \begin{itemize}
        \item the composition $\pi_1 \circ \beta:[0,1]\to \RR^n$ is semi-algebraic,

        \item the composition $\pi_2\circ \beta:[0,1]\to \RR^m$ is non-constant,

        \item $\pi_2(\beta(0)) \in \pi_2(\Sigma)$,

        \item if $\calS$ admits analytic cell decomposition, then $\beta|_{(0,1)}$ is real-analytic.
    \end{itemize}
\end{prop}

Throughout this article we work in the structure $\RR_{\mathrm{an},\exp}$ generated by the restricted real-analytic functions together with the real exponential; ``definable'' will always mean definable in $\RR_{\mathrm{an},\exp}$. This structure is o-minimal, by Wilkie~\cite{Wilkie} and van~den~Dries--Miller~\cite{vdDriesMiller}, and it admits analytic cell decomposition. In particular the last conclusion of \Cref{prop:HP} is always available to us, so the arc $\beta$ may be taken real-analytic on $(0,1)$.

We now put the uniformization into this definable setting. Recall the uniformization $u_{g,a}:\CC^g\times\gothH_g\to\gothA_g^{\mathrm{an}}$ of the universal abelian variety. By Peterzil--Starchenko~\cite{PS_Definability13}, there is a semi-algebraic subset 
\[
\gothF:= \left\{(\tau a+b,\tau):\tau\in\mathfrak F_g(3),\ (a,b)\in[-1/2,1/2]^{2g} \right\} \subseteq\CC^g\times\gothH_g,
\]
where \(\mathfrak F_g(3)\) is a suitable fundamental set for the level group, such that the restriction $u_{g,a}|_\gothF:\gothF\to\gothA_g^{\mathrm{an}}$ is surjective and definable in $\RR_{\mathrm{an},\exp}$, i.e. $\gothF$ is a definable fundamental set for the uniformization. Moreover the universal uniformized Betti map $\widetilde{\gothb}:\CC^g\times\gothH_g\to\RR^{2g}$ of \eqref{eq:universalUniformizedBetti} is definable, being the composition of the inverse of the fiber-wise linear isomorphism \eqref{eqn:realIsoUniversalCover} with a coordinate projection.

For an algebraic subvariety $\calX\subseteq\gothA_g$, we assemble the Betti coordinates on $\calX$ into the definable family
\begin{equation}\label{eq:defFamily}
F(\calX):=\Bigl\{(\psi,w,z)\in\Mat_{2g}(\RR)\times\RR^{2g}\times\gothF:\ z\in\bigl(u_{g,a}|_\gothF\bigr)^{-1}(\calX^{\mathrm{an}}),\ \psi\,w=\psi\,\widetilde{\gothb}(z)\Bigr\},
\end{equation}
where $\psi\in\Mat_{2g}(\RR)$ acts on $\RR^{2g}$ as a matrix, so that $F(\calX)$ is parametrized by $\RR^{l}$ with $l=(2g)^2$. Each condition defining $F(\calX)$ is definable, so $F(\calX)$ is a definable family. In the application the parameter $\psi$ will be specialized to the rational representation $M_r(\phi)$ of an endomorphism $\phi\in\End(\calA/S)$, viewed as a real matrix. Then the fiber $F(\calX)_\psi$ then records the points of $\calX$ whose Betti coordinates agree with those of a given point after applying $\phi$, that is, the Betti coordinates on $\calX$ up to the endomorphism $\phi$.

\subsection{Betti relations force degeneracy}\label{subsec:algdep}
Let $S\subseteq\bA_g$ be a locally closed subvariety, and let $\calA=\gothA_g|_S$ be the induced abelian scheme. Fix a connected component $\widetilde S$ of $u_g^{-1}(S^{\mathrm{an}})$ as in \Cref{subsec:rrr}. We write $\Pi(\tau):=(\tau\ \ 1_g)$ and $g_\phi:=\rdim\ker\phi$ for $\phi\in\End(\calA/S)$.

\begin{prop}\label{prop:transcendence}
Let $\calX\subseteq\calA$ be an irreducible subvariety, and let
\[
\beta:[0,1]\longrightarrow u_{g,a}^{-1}(\calX^{\mathrm{an}}) \cap(\CC^g\times\widetilde S)
\]
be an arc that is real-analytic on $(0,1)$. Let $\phi\in\End(\calA/S)$ be nonzero. Suppose that there is a semialgebraic set $\gamma\subseteq\RR^{2g}$ of dimension at most $1$ satisfying
\[
M_r(\phi)\bigl(\widetilde{\gothb}(\beta([0,1]))\bigr) \subseteq\gamma.
\]
Then $u_{g,a}(\beta([0,1]))\subseteq \calX^{\mathrm{deg}}(g_\phi+1)$.
\end{prop}

\begin{proof}
Write
\[
R:=M_r(\phi)=\begin{pmatrix}
* & B \\
* & D
\end{pmatrix} \in\Mat_{2g}(\ZZ),
\qquad
M_a(\tau):=\tau B+D \quad\bigl(\tau\in\operatorname{Sym}_g(\CC)\bigr).
\]
The second blocks of \eqref{eqn:relativeCompatibility} identify $M_a(\phi,\tau)$ with $M_a(\tau)$ for $\tau\in\widetilde S$. With this identification, the same compatibility identity reads
\begin{equation}\label{eqn:transcendence-compatibility}
M_a(\tau)\Pi(\tau)=\Pi(\tau)R \qquad(\tau\in\widetilde S).
\end{equation}
Since $\Pi(\tau):\RR^{2g}\to\CC^g$ is a real-linear isomorphism and $M_a(\tau)$ represents the differential of $\phi$, we have
\begin{equation}\label{eqn:transcendence-rank}
\operatorname{rank}_{\CC}M_a(\tau) =\frac12\operatorname{rank}_{\RR}R
=g-g_\phi \qquad(\tau\in\widetilde S).
\end{equation}
In particular, the fibers of $\ker\phi$ have constant dimension $g_\phi$.

Put $\calM:=\beta((0,1))$ and $\calX_0:=\calX^{\mathrm{Zar}}$, with the latter closure taken in $\gothA_g$. Let $Z$ be the smallest closed complex analytic subset of $\CC^g\times\gothH_g$ containing $\calM$. Then $Z$ is irreducible. Nonconstancy and continuity of $\beta$ give
\[
\dim Z\geq1,\qquad \beta([0,1])\subseteq Z \subseteq u_{g,a}^{-1}(\calX_0^{\mathrm{an}}).
\]
Set
\[
W:=\calM^{\mathrm{Zar}}=Z^{\mathrm{Zar}}, \qquad T:=p(\calM)^{\mathrm{Zar}}=p(W)^{\mathrm{Zar}},
\]
where these closures are taken in $\CC^g\times\operatorname{Sym}_g(\CC)$ and $\operatorname{Sym}_g(\CC)$, respectively.

Put $r:=g-g_\phi$. By \eqref{eqn:transcendence-rank}, all $(r+1)$-minors of $M_a(\tau)$ vanish on $p(\calM)$, hence on $T$. Therefore
\[
T^\circ:=\{\tau\in T:\operatorname{rank}M_a(\tau)=r\}
\]
is a Zariski-open subset of $T$ containing $p(\calM)$. Let $C_\gamma:=\gamma^{\mathrm{Zar}}$ denote the complex Zariski closure of $\gamma$ in $\CC^{2g}$. Then $\dim C_\gamma\leq1$ by \cite[\S2.8]{BochCosRoy}.

Consider the incidence set
\[
J:=\bigl\{(\tau,c,z)\in T^\circ\times C_\gamma\times\CC^g: M_a(\tau)z=\Pi(\tau)c\bigr\}.
\]
Since $M_a(\tau)$ and $\Pi(\tau)$ have polynomial entries, $J$ is closed algebraic in $T^\circ\times C_\gamma\times\CC^g$, and locally closed in the ambient affine space. Every nonempty fiber of $J\to T^\circ\times C_\gamma$ is an affine translate of $\ker M_a(\tau)$, of dimension $g_\phi$. Consequently,
\begin{equation}\label{eqn:transcendence-incidence-bound}
\dim J\leq\dim T+\dim C_\gamma+g_\phi \leq\dim T+g_\phi+1.
\end{equation}
Let $E$ be the image of $J$ under $(\tau,c,z)\mapsto(z,\tau)$. By Chevalley's theorem, $E$ is constructible, so
\[
\dim E^{\mathrm{Zar}}=\dim E\leq\dim J.
\]
For $(z,\tau)\in\calM$, put $v:=\widetilde{\gothb}(z,\tau)$, so that $z=\Pi(\tau)v$. By hypothesis $Rv\in\gamma$, and \eqref{eqn:transcendence-compatibility} gives
\[
M_a(\tau)z=\Pi(\tau)Rv.
\]
Since $\tau\in p(\calM)\subseteq T^\circ$, the triple $(\tau,Rv,z)$ lies in $J$. Thus $\calM\subseteq E$ and $W\subseteq E^{\mathrm{Zar}}$. Using \eqref{eqn:transcendence-incidence-bound}, we obtain
\[
\rdim W=\dim W-\dim T\leq g_\phi+1.
\]

Finally, set
\[
Y_0:=u_{g,a}(Z)^{\mathrm{Zar}}\subseteq\calX_0, \qquad Y:=Y_0\cap\calX.
\]
The variety $Y_0$ is irreducible because $Z$ is irreducible, and positive-dimensional because $u_{g,a}$ is a local biholomorphism. Since $\calX$ is open in $\calX_0$, $Y$ is a nonempty open subset of $Y_0$, closed in $\calX$. In particular, $Y$ is irreducible, $\dim Y=\dim Y_0\geq1$, and
\[
Y^{\mathrm{biZar}} =Y_0^{\mathrm{biZar}} =u_{g,a}(Z)^{\mathrm{biZar}}.
\]
Using \Cref{rmk:ws-vs-mixed-shimura} and \Cref{prop:AxScCor}, we obtain
\begin{align*}
\delta_v(Y) &=\rdim Y^{\mathrm{ws}}-\dim Y\\
&=\rdim u_{g,a}(Z)^{\mathrm{biZar}} - \dim u_{g,a}(Z)^{\mathrm{Zar}}\\
&\leq\rdim W-\dim Z \leq g_\phi.
\end{align*}
Hence $Y\subseteq\calX^{\mathrm{deg}}(g_\phi+1)$. Since $\beta([0,1])\subseteq Z$ and $u_{g,a}(\beta([0,1]))\subseteq\calX$ by hypothesis, we conclude that
\[
u_{g,a}(\beta([0,1])) \subseteq Y \subseteq\calX^{\mathrm{deg}}(g_\phi+1),
\]
as required.
\end{proof}

\section{Complexity and large Galois orbits}
In this section we first define a general notion of special subscheme of an abelian variety $A$ relative to a subgroup $\Gamma\subseteq\End(A)$, which we call $\Gamma$-special. We then give an explicit upper bound for the degree of the $\Gamma$-special closure of a closed point in \Cref{thm:LGO}, and use it to establish the large-Galois-orbit statement required by the o-minimal counting theorem.

\subsection{\texorpdfstring{$\Gamma$}{Gamma}-special subschemes and complexity}\label{subsec:complexity}

Let $(A,L)$ be a polarized abelian variety over a number field $K$, and let $\Gamma\subseteq\End(A)$ be a nonzero subgroup. One class of subgroups is most relevant to us: if $A$ is a fiber of an abelian scheme $\calA/S$ over an irreducible variety $S$, then $\End(\calA/S)$ specializes injectively into $\End(A)$; we call the image a \emph{specialized endomorphism algebra}.

\begin{defn}
For a subset $\Gamma_0$ of $\Gamma$, the group subscheme $Z(\Gamma_0):=\bigcap_{\gamma\in \Gamma_0}\ker\gamma$ is called a \emph{$\Gamma$-special subscheme} of $A$. By convention, $Z(\emptyset)=A$.
\end{defn}

It follows from the definition that an arbitrary intersection of $\Gamma$-special subschemes is again $\Gamma$-special. We may therefore define the \emph{$\Gamma$-special closure} of a subset $\Sigma\subseteq A$, written $\langle\Sigma\rangle_\Gamma$, as the intersection of all $\Gamma$-special subschemes containing $\Sigma$. Writing
\[
I_\Gamma(\Sigma):=\{\gamma\in\Gamma:\Sigma\subseteq\ker\gamma\}
\]
for the \emph{relation subgroup} of $\Sigma$, we have $\langle\Sigma\rangle_\Gamma=Z(I_\Gamma(\Sigma))$.

\begin{lemma}\label{lem:compareGamma}
    Let $\Gamma'\subseteq\Gamma$ be a subgroup of finite index. For any $\Sigma\subseteq A$ we have $\langle\Sigma\rangle_\Gamma\subseteq\langle\Sigma\rangle_{\Gamma'}$ and $\langle\Sigma\rangle_\Gamma^0=\langle\Sigma\rangle_{\Gamma'}^0$.
\end{lemma}
\begin{proof}
    Let $n:=[\Gamma:\Gamma']$, so that $n\Gamma\subseteq\Gamma'$ and hence
    \[
    nI_\Gamma(\Sigma)\subseteq I_{\Gamma'}(\Sigma)\subseteq I_\Gamma(\Sigma).
    \]
    Applying $Z(\cdot)$, which reverses inclusions, gives $\langle\Sigma\rangle_\Gamma\subseteq\langle\Sigma\rangle_{\Gamma'}\subseteq Z(nI_\Gamma(\Sigma))$. For $\gamma\in\End(A)$ one has $\ker(n\gamma)=\gamma^{-1}(A[n])$, a finite union of translates of $\ker\gamma$, so $\ker(n\gamma)^0=\ker(\gamma)^0$. Hence $Z(nI_\Gamma(\Sigma))$ and $Z(I_\Gamma(\Sigma))$ have the same identity component, and the three terms above agree after passing to identity components.
\end{proof}

% it is not true that a single element of $I_\Gamma(\Sigma)$ can give the same identity component in general

\begin{defn} \label{def:Gamma_complexity}
    For a closed point $x$ of $A$, the \emph{$\Gamma$-complexity} of $x$ (or simply \emph{complexity}, when $\Gamma$ is clear) is the degree of its $\Gamma$-special closure with respect to $L$:
    \[
    \Delta_\Gamma(x):=\deg\bigl(\langle x\rangle_\Gamma\bigr).
    \]
\end{defn}
 
\subsection{Extremal heights along \texorpdfstring{$\Gamma$}{Gamma}}
We attach to $\Gamma$ two functions built from the N\'eron--Tate height $\hat h_L$ on $A$. Recall that $\hat h_L$ is the N\'eron--Tate height of $\tfrac12(L+[-1]^*L)$, and that $\lVert\cdot\rVert$ denotes the Weil metric on $\End(A)$. The first measures the largest value of the quadratic form $\gamma\mapsto\hat h_L(\gamma(x))$ relative to $\lVert\cdot\rVert$:
\begin{equation}\label{eq:GammaSingularValue}
    \hat h_\Gamma(x):=\sup_{\gamma\in\Gamma\setminus\{0\}} \frac{\hat h_L(\gamma(x))}{\lVert\gamma\rVert^2}, \qquad x\in A(\bar K).
\end{equation}

\begin{lemma}
$\hat h_\Gamma(x)$ is finite for every $x \in A(\bar K)$.
\end{lemma}
\begin{proof}
    Enlarging $\Gamma$ only increases the supremum, so we may assume $\Gamma=\End(A)$. Fix $x\in A(\bar K)$ and let
    \[
    T_x:\End_\RR(A)\to A(\bar K)\otimes_\ZZ\RR
    \]
    be the $\RR$-linear extension of $\gamma\mapsto\gamma(x)$. Its image lies in the finite-dimensional subspace spanned by $\gamma(x)$ for $\gamma$ in a basis of $\End(A)$, on which $\hat h_L$ restricts to a positive semidefinite quadratic form. Hence $Q_x:=\hat h_L\circ T_x$ is a quadratic form on the finite-dimensional real vector space $\End_\RR(A)$. Writing $\mathcal{S}$ for the unit sphere of $(\End_\RR(A),\lVert\cdot\rVert)$, homogeneity gives
    \[
    \hat h_\Gamma(x)=\sup_{\gamma\in\Gamma\setminus\{0\}} \frac{Q_x(\gamma)}{\lVert\gamma\rVert^2}\le\sup_{\gamma\in\mathcal{S}}Q_x(\gamma),
    \]
    which is finite by continuity of $Q_x$ and compactness of $\mathcal{S}$.
\end{proof}

The second function measures the smallest nonzero value of the same form:
\begin{equation}\label{eq:GammaGapDef}
\tilde h_\Gamma(x):=\inf\bigl\{\hat h_L(\gamma(x)): \gamma\in\Gamma,\ \gamma(x)\text{ non-torsion}\bigr\}, \qquad x\in A(\bar K),
\end{equation}
with the convention $\inf\emptyset=\infty$. It will be convenient to record their normalized ratio
\begin{equation}\label{eq:RGammaDef}
R_\Gamma(x):=\max\Bigl\{1,\ \frac{\hat h_\Gamma(x)}{\tilde h_\Gamma(x)}\Bigr\}.
\end{equation}
%The truncation at $1$ is necessary: for $\Gamma=\ZZ\cdot\mathrm{id}$ one has $\lVert\mathrm{id}\rVert^2=2g$, so $\hat h_\Gamma(x)=\hat h_L(x)/2g$ while $\tilde h_\Gamma(x)=\hat h_L(x)$, and the ratio is $1/(2g)<1$.

\begin{defn}
    A point $x\in A(\bar K)$ is \emph{$\Gamma$-torsion} if $\gamma(x)$ is torsion for every $\gamma\in \Gamma$.
\end{defn}
Thus $x$ is $\Gamma$-torsion if and only if $\hat h_\Gamma(x)=0$, if and only if $\tilde h_\Gamma(x)=\infty$.

\subsection{An explicit bound for the complexity}
Let $(A,L)$ be as above. Two quantities attached to the Mordell--Weil group $A(K)$ will appear. The first is the size of the torsion subgroup,
\begin{equation}\label{eq:NKDef}
N_K:=\# A(K)_{\mathrm{tors}},
\end{equation}
and the second is a Lehmer-type lower bound: writing $A(K)_{\mathrm{nt}}:=A(K)\setminus A(K)_{\mathrm{tors}}$,
\begin{equation}\label{eq:hKDef}
h_K:=\min_{x\in A(K)_{\mathrm{nt}}}\hat h_L(x).
\end{equation}
We use the convention $h_K=+\infty$ if $A(K)_{\mathrm{nt}}=\varnothing$.

\begin{theorem}[Gaudron--R\'emond]\label{thm:GR}
    Let $(A,L)$ be a polarized abelian variety of dimension $g$ over a number field $K$, let $D:=[K:\QQ]$, and let $h_F(A)$ be the stable Faltings height. Then
    \begin{enumerate}
        \item $N_K \leq c(g,D,h_F(A))^g$,
        \item $h_K^{-1} \leq D \cdot c(g,D,h_F(A))^{2g}$,
    \end{enumerate}
    where $c(g, D, h_F(A)) := (6g)^8 D\max\{1, h_F(A), \log D\}$.
\end{theorem}
\begin{proof}
    See \cite[Thm.~1.2 and Thm.~1.3]{GR}.
\end{proof}

We next record an elementary estimate for independent relations, motivated by \cite{Masser_Linear}. Given $P_1,\ldots,P_m\in A(K)$, with $m\geq1$, we denote by
\[
\Lambda(P_1,\ldots,P_m):=\left\{(c_1,\ldots,c_m)\in\ZZ^m:
\sum_{i=1}^m c_iP_i=0\right\}
\]
the relation group. As before, $|\cdot|$ denotes the standard Euclidean norm on $\RR^m$; the notation $\lVert\cdot\rVert$ remains reserved for the Weil norm on endomorphisms in this subsection.

Our proof follows Masser's geometry-of-numbers argument, using the Euclidean norm in place of the maximum norm. We exploit the quadratic structure of the N\'eron--Tate height to replace his cube-slicing estimate by a direct covolume computation, and then apply Minkowski's second theorem as in his proof. Since we only require independent relations spanning the relation space over $\QQ$, we omit the final passage to a $\ZZ$-basis. These modifications yield the explicit factor $m^{m/2}$ in the bound below.

\begin{lemma}\label{lem:relations}
    Let $(A,L)$ be a polarized abelian variety over a number field $K$. Consider $P_1,\ldots,P_m\in A(K)$, with $m\geq1$, and denote by $G$ the subgroup they generate. Let $\Lambda:=\Lambda(P_1,\ldots,P_m)$ of rank $l:=\operatorname{rank}\Lambda >0$. There exist linearly independent vectors $\vec v_1,\ldots,\vec v_l\in\Lambda$, spanning $\Lambda\otimes\QQ$, such that:
\begin{enumerate}
    \item in the case $l < m$, if $\hat h_L(P_i)\leq q$ for every $i$ and every non-torsion element of $G$ has canonical height at least $h_0>0$, then
    \[
    \prod_{j=1}^l|\vec v_j| \leq m^{\frac{m}{2}}N_K\left(\frac{q}{h_0}\right)^{\frac{m-l}{2}},
    \]

    \item in the case $l=m$, then
    \[
    \prod_{j=1}^m|\vec v_j|\leq m^{\frac{m}{2}}N_K.
    \]
\end{enumerate}
\end{lemma}

\begin{proof}
First suppose $l < m$. Let $\mathcal M:=G/G_{\mathrm{tors}}$, viewed as a lattice in $V:=\mathcal M\otimes_\ZZ\RR$, endowed with the Euclidean norm whose square is the canonical height. In this proof, covolumes and adjoints are taken with respect to this metric on $V$ and the standard Euclidean metric on $\RR^m$. Consider the surjective linear map
\[
f:\RR^m\longrightarrow V, \qquad e_i\longmapsto (P_i+ G_{\mathrm{tors}}) \otimes 1.
\]
Its integral kernel $\Lambda_{\mathrm{sat}}:=\ker f\cap\ZZ^m$ is the saturation of $\Lambda$. We set $r:= \dim{V}=m-l$ and denote by $s(f)$ the product of the $r$ nonzero singular values $\sigma_1,\ldots,\sigma_r$ of $f$. The restriction of $f$ to $\ZZ^m$ gives an exact sequence
\[
\begin{tikzcd}
    0 \arrow{r} & \Lambda_{\mathrm{sat}} \arrow{r} & \ZZ^m \arrow{r}{f} & \mathcal M \arrow{r} & 0.
\end{tikzcd}
\]
Denote by $\mathcal Q$ the projection of $\ZZ^m$ onto $(\ker f)^\perp$. Completing a basis of the primitive lattice $\Lambda_{\mathrm{sat}}$ to a basis of $\ZZ^m$ and computing volumes in the orthogonal decomposition $\RR^m=\ker f\oplus(\ker f)^\perp$, we obtain $\operatorname{covol}(\mathcal Q)=\operatorname{covol}(\Lambda_{\mathrm{sat}})^{-1}$. Since applying $f$ multiplies this covolume by $s(f)$, we obtain 
\[
\operatorname{covol}(\Lambda_{\mathrm{sat}}) =\frac{s(f)}{\operatorname{covol}(\mathcal M)}.
\]
Since $\operatorname{Tr}(ff^*) = \sum_{i=1}^m\hat h_L(P_i)\leq mq$, the arithmetic--geometric mean inequality gives
\[
s(f)=\left( \prod_{j=1}^r \sigma_j^2\right)^{\frac{1}{2}} \le \left( \frac{\sum_{j=1}^r \sigma_j^2}{r}\right)^{\frac{r}{2}} \le \left( \frac{mq}{r} \right)^{\frac{r}{2}}.
\]
Moreover, since every nonzero vector of $\mathcal M$ has norm at least $\sqrt{h_0}$, Minkowski's first theorem implies
\[
\operatorname{covol}(\mathcal M) \geq\frac{V_r}{2^r}h_0^{r/2},
\]
where $V_j$ denotes the volume of the Euclidean unit ball in $\RR^j$. Since evaluation at $(P_1,\ldots,P_m)$ induces an isomorphism $\Lambda_{\mathrm{sat}}/\Lambda\simeq G_{\mathrm{tors}}$ and using $V_j\ge (2/\sqrt j)^j$, we finally obtain
\[
\operatorname{covol}(\Lambda) = \#G_{\mathrm{tors}} \cdot \operatorname{covol}(\Lambda_{\mathrm{sat}}) \le N_K m^\frac{r}{2} \left( \frac{q}{h_0} \right)^\frac{r}{2}.
\]
Choose independent vectors $\vec v_1,\ldots,\vec v_l$ realizing the successive minima of $\Lambda$ in its real span. These are $l$ independent vectors spanning $\Lambda\otimes\QQ$ and, by Minkowski's second theorem, they satisfy
\[
\prod_{j=1}^l|\vec v_j| \le\frac{2^l}{V_l}\operatorname{covol}(\Lambda) \le l^{\frac{l}{2}}N_K m^{\frac{r}{2}}\left(\frac{q}{h_0}\right)^{\frac{r}{2}} \le m^{\frac{m}{2}}N_K \left(\frac{q}{h_0}\right)^{\frac{r}{2}}.
\]

If $l=m$, then $G$ is finite and $[\ZZ^m:\Lambda]=\#G\leq N_K$. Applying Minkowski's second theorem to this full-rank lattice provides independent vectors with
\[
\prod_{j=1}^m|\vec v_j| \le \frac{2^m}{V_m}\operatorname{covol}(\Lambda) \le m^{\frac{m}{2}}N_K.
\]
The proof is now complete. 
\end{proof}

The following is the main result of this section, motivated by \cite[Prop.~6.1]{BC_CurvesinAbelianScheme} of the first and second-named authors.

\begin{theorem}\label{thm:LGO}
Let $(A,L)$ be a polarized abelian variety of dimension $g\geq1$ over a number field $K$, let $\Gamma\subseteq\End(A)$ be a subgroup of rank $m\geq1$, and let $\mu_m$ be the $m$-th successive minimum of the Weil norm $\lVert\cdot\rVert$ on $\Gamma$. Set $\mu_m^+:=\max\{1,\mu_m\}$. For every $x\in A(K)$ one has
\begin{equation}\label{eq:LGOrational}
    \Delta_\Gamma(x)\le\Bigl(\frac{3^g\deg_L(A)}{2g}\, m^{2m}(\mu_m^+)^{2m}N_K^2R_\Gamma(x)^{m-1}\Bigr)^g.
\end{equation}
Consequently, for every closed point $x\in A(\bar K)$ one has
\begin{equation}\label{eq:LGO}
    \Delta_\Gamma(x)\le\Bigl(\frac{3^g\deg_L(A)}{2g}\, m^{2m}(\mu_m^+)^{2m}D^{m-1}c^{2gm}\max\{1,\hat h_\Gamma(x)\}^{m-1}\Bigr)^g,
\end{equation}
where $D=[K(x):\QQ]$ and $c=c(g,D,h_F(A))$ is as in \Cref{thm:GR}.
\end{theorem}

\begin{proof}
Let $\mu_1\leq\cdots\leq \mu_m$ be the successive minima of the Weil norm $\lVert \cdot \rVert$ on $\Gamma$. Choose linearly independent elements $\gamma_1,\ldots,\gamma_m\in\Gamma$ spanning $\Gamma\otimes\QQ$, with $\lVert\gamma_i\rVert=\mu_i$ for all $i$. Let $\Gamma'$ be the finite-index sublattice of $\Gamma$ generated by $\gamma_1,\ldots,\gamma_m$. We keep $\Gamma$ fixed in the height quantities below.

Define $P_i:=\gamma_i(x)$ for $1\leq i\leq m$. We denote by $G\subseteq A(K)$ the subgroup generated by $P_1,\ldots,P_m$ and by $\Lambda=\Lambda(P_1,\cdots,P_m) \subseteq \ZZ^m$ the relation group. Set $l:=\operatorname{rank}\Lambda$ and $r:=m-l$. The result is trivial if $l=0$, hence we assume $l>0$.

Assume $l<m$ and set
\[
q:=\mu_m^2\hat h_\Gamma(x),\qquad h_0:=\tilde h_\Gamma(x)>0.
\]
By \eqref{eq:GammaSingularValue}, we have $\hat h_L(P_i)\le \lVert \gamma_i\rVert^2 \hat h_\Gamma(x)\le q$. Moreover, every non-torsion element of $G$ has height at least $h_0$, since it belongs to $\Gamma'(x)\subseteq\Gamma(x)$. Thus, by \Cref{lem:relations}, there exist linearly independent vectors $\vec v_1,\ldots,\vec v_l\in\Lambda$, spanning $\Lambda\otimes\QQ$, such that
\[
\prod_{j=1}^l|\vec v_j| \le m^{\frac{m}{2}}N_K \left(\frac{\mu_m^2\hat h_\Gamma(x)} {\tilde h_\Gamma(x)}\right)^{\frac{r}{2}}.
\]
For $1\leq j\leq l$, set $\tau_j:=\sum_{i=1}^m(\vec v_j)_i\gamma_i$. These are nonzero elements of $I_{\Gamma'}(x)$ and span $I_{\Gamma'}(x)\otimes\QQ$. Since $|\vec v_j|\ge 1$, the triangle inequality and Cauchy--Schwarz give
\[
\max\{1,\lVert\tau_j\rVert\} \leq\sqrt m\,\mu_m^+|\vec v_j|.
\]
Since $l<m$, we obtain
\begin{equation}\label{eqn:boundMink}
\prod_{j=1}^l\max\{1,\lVert\tau_j\rVert\} \le m^{\frac{m+l}{2}}(\mu_m^+)^lN_K \left(\frac{\mu_m^2\hat h_\Gamma(x)}{\tilde h_\Gamma(x)}\right)^{\frac{r}{2}}\le m^m(\mu_m^+)^mN_KR_\Gamma(x)^{\frac{m-1}{2}}.
\end{equation}
In the case $l=m$, the point $x$ is $\Gamma$-torsion. Choose $\vec v_1,\ldots,\vec v_m$ using \Cref{lem:relations} and define $\tau_j$ as above. Since $R_\Gamma(x)=1$, the same estimate gives \eqref{eqn:boundMink}.

Since $I_{\Gamma'}(x)$ has finite index in $I_\Gamma(x)$, the elements $\tau_1,\ldots,\tau_l$ also generate a finite-index subgroup of $I_\Gamma(x)$. By \Cref{lem:compareGamma} and for dimension reasons, there exists a subsequence $\tau_1',\ldots,\tau_k'$ of $\tau_1,\ldots,\tau_l$, with $k\leq g$, such that
\[
\langle x\rangle_\Gamma^0 =\left(\bigcap_{j=1}^l\ker\tau_j\right)^0 =\left(\bigcap_{j=1}^k\ker\tau_j'\right)^0.
\]
The corresponding group schemes are nested and have the same dimension, so
\[
\Delta_\Gamma(x) =\deg_L\langle x\rangle_\Gamma \leq\deg_L\left(\bigcap_{j=1}^l\ker\tau_j\right) \leq\deg_L\left(\bigcap_{j=1}^k\ker\tau_j'\right).
\]
Notice that the line bundle $L^{\otimes3}$ is very ample and that both sides of \Cref{eqn:boundMink} do not change if we replace the polarization by a multiple of $L$, e.g. $L^{\otimes3}$. By B\'ezout's theorem and \Cref{prop:degreeBound},
\begin{align*}
\Delta_\Gamma(x) &\le \deg_L\left(\bigcap_{j=1}^l \ker\tau_j\right) \le \deg_{L^{\otimes3}}\left(\bigcap_{j=1}^k \ker\tau'_j\right) \le \prod_{j=1}^k\deg_{L^{\otimes3}}(\ker\tau'_j) \\
&\le \left( \frac{\deg_{L^{\otimes3}}(A)}{2g} \right)^k \prod_{j=1}^k\max\{1,\lVert\tau'_j\rVert\}^{2g} \le \left( \frac{3^g \deg_L(A)}{2g} \right)^g \left( \prod_{j=1}^l\max\{1,\lVert\tau_j\rVert\} \right)^{2g} \\
&\leq \left( \frac{3^g \deg_L(A)}{2g} m^{2m}(\mu_m^+)^{2m}N_K^2R_\Gamma(x)^{m-1} \right)^g.
\end{align*}
This proves \eqref{eq:LGOrational}. Finally, by \Cref{thm:GR}, we have $N_{K(x)}\leq c^g$, where $D:=[K(x):\QQ]$ and $c:=c(g,D,h_F(A))$. If $x$ is not $\Gamma$-torsion, then $\tilde h_\Gamma(x)\ge h_{K(x)}$ and $h_{K(x)}^{-1}\leq Dc^{2g}$, whence
\[
R_\Gamma(x) \le \max\{1,\hat h_\Gamma(x)Dc^{2g}\} \leq\max\{1,\hat h_\Gamma(x)\}\,Dc^{2g}.
\]
Notice that the last bound also holds when $x$ is $\Gamma$-torsion, since then $R_\Gamma(x)=1$ and $Dc^{2g}\ge 1$. Applying \eqref{eq:LGOrational} over $K(x)$ proves \eqref{eq:LGO}. The proof is now complete.
\end{proof}

\subsection{Large Galois orbits}
We now apply \Cref{thm:LGO} uniformly in an abelian scheme. On a set of bounded total height, \eqref{eq:LGO} gives a polynomial lower bound for the size of the Galois orbit of a point in terms of its complexity.

Let $S$ be a normal quasi-projective variety over a number field $K$, and let $\pi:\calA\to S$ be an abelian scheme of relative dimension $g\geq 1$, polarized by a relatively ample line bundle $\calL$. Fix a Weil height $h_S:S(\bar K)\to\RR_{\geq 0}$ associated with a projective embedding of $S$, and let $\hat h_\calL:\calA(\bar K)\to\RR_{\geq 0}$ be the fiber-wise N\'eron--Tate height. The \emph{total height} is
\[
h_{\mathrm{tot}} :=\hat h_\calL+h_S\circ\pi: \calA(\bar K)\to\RR_{\geq 0}.
\]
We take $\Gamma:=\End(\calA/S)$, viewed as a common subgroup of the endomorphism algebras of the fibers via the specialization injections $\End(\calA/S)\hookrightarrow\End(\calA_s)$. For a point $x\in\calA(\bar K)$ over $s\in S(\bar K)$, the $\Gamma$-complexity of $x$ is computed in the polarized fiber $(\calA_s,\calL_s)$.

\begin{prop}\label{prop:LGO}
Let $K,S,\calA,\calL,h_{\mathrm{tot}},\Gamma$ be as above, and let $\delta>0$. There exist constants $c_0,\epsilon_0>0$, depending only on these data and on $\delta$, such that every $x\in\calA(\bar K)$ with $h_{\mathrm{tot}}(x)<\delta$ satisfies
\[
[K(x):\QQ]\geq c_0\Delta_\Gamma(x)^{\epsilon_0}.
\]
\end{prop}

\begin{proof}
The Weil norm on $\Gamma$ is independent of $s$ so that its successive minima is constant in each fiber. The quadratic pullback map $\gamma\mapsto[\gamma^*\calL]$ and compactness of the unit ball give a uniform comparison (see \cite[Prop.~2.9]{Ge_IntersectI})
\[
\hat h_\Gamma(x)\leq c_1\hat h_\calL(x)\leq c_1\delta.
\]
Moreover, by \cite[\S5.5]{YuanZhang_ALB}, we have
\[
h_F(\calA_s)\leq c_2\max\{1,h_S(s)\} \leq c_2\max\{1,\delta\}.
\]

Put $D:=[K(x):\QQ]$. Since $K(s)\subseteq K(x)$, we may apply \Cref{thm:LGO} to $\calA_s$ over $K(x)$. The bounds above imply
\[
c(g,D,h_F(\calA_s))\ll D(1+\log D),
\]
while all remaining factors in \eqref{eq:LGO}, apart from $D$, are uniformly bounded. Absorbing the logarithmic factor into a power of $D$, we obtain
\[
\Delta_\Gamma(x)\leq C D^N
\]
for constants $C,N>0$ depending only on the fixed data and on $\delta$. Taking $\epsilon_0=1/N$ and $c_0=C^{-1/N}$ proves the assertion.
\end{proof}

\section{Unlikely intersections and nondegeneracy}

In this section, we prove the main result of the article. 

\begin{theorem}\label{thm:main}
    Let $t\in \ZZ_{>0}$. Let $S$ be a normal quasi-projective variety and let $\calA/S$ be an abelian scheme, both defined over a number field $K$. Let $\calX$ be a positive dimensional subvariety of $\calA$ over $K$ and let $\calA_{<t}$ denote the union of all flat subgroup schemes of $\calA$ of relative dimension $<t$. Then, the intersection $(\calX \setminus \calX^\textnormal{deg}(t)) \cap \calA_{<t}$ is finite. In particular, if $\calX$ is $t$-nondegenerate, the intersection $\calX\cap \calA_{<t}$ is not Zariski dense in $\calX$.
\end{theorem}

We first reduce the problem to the case where $S$ is a locally closed smooth subvariety of the moduli space. We prove this essential case following the Pila--Zannier strategy using the Pila--Wilkie counting theorem. The key ingredients in the proof are \Cref{prop:LGO}, the bounded height theorem of the third-named author~\cite{Ge_IntersectI}, and the Ax--Schanuel type result~\Cref{prop:AxScCor}.

\subsection{Reductions}
First, we make some elementary observations:
\begin{enumerate}
    \item \Cref{thm:main} is equivalent for two isogenous abelian schemes.
    
    \item \Cref{thm:main} for $\calA/S$ is equivalent to that for $\calA'/S'$ where $S'\rightarrow S$ is a finite \'etale cover and $\calA'=\calA\times_S S'$.
    
    \item Suppose $\calB\hookrightarrow \calA$ is an abelian subscheme. If \Cref{thm:main} holds for $\calA/S$, then it also holds for $\calB/S$, since $\calB\cap\calA_{<t}=\calB_{<t}$.
\end{enumerate}

If $\calB$ is an isogeny component of $\calA$, then by the observations (1) and (3), it suffices to prove the result for $\calA$. By Zarhin's trick~\cite{Zarhin_AV}, we know $(\calA\times_S\calA^\dual)^4$ can be principally polarized. Hence, without loss of generality, we can reduce to assume that $\calA/S$ is principally polarized. In the meantime, thanks to the freedom provided by (2), we may take a finite \'etale base change to assume that $\calA/S$ can be equipped with a level-$3$-structure which can be chosen to be symplectic by our convention $\QQ[\zeta_3]\subseteq K$; see \cite[Lem.~2.2]{Ge_IntersectI}. Overall, it boils down to the case where $\calA/S$ is \emph{principally polarized with a symplectic level-$3$-structure}. In particular, there is a modular map $\iota_S:S\rightarrow \AAA_g$ to the moduli space $\AAA_g$ of principally polarized abelian varieties with symplectic level-$3$-structures over $K$. 

The final reduction is to the case where $S$ is a locally closed smooth subvariety of $\AAA_g$:

\begin{lemma}\label{lem:modularFinitenessReduction}
    Let $\AAA_g$ be the moduli space of principally polarized abelian varieties with symplectic level-$3$-structures over a number field $K$. If \Cref{thm:main} holds for any locally closed smooth subvariety $T\subseteq\AAA_g$ and the induced abelian scheme, then it holds for any  principally polarized $\calA/S$ over $K$ with a symplectic level-$3$-structure.
\end{lemma}

\begin{proof}
Induct on $\dim\calX$. The $\dim\calX=0$ case is trivial.

Let $\gothA_g\to\AAA_g$ be the universal abelian scheme, and let $\iota:\calA\to\gothA_g$ be the modular map. After restricting to a dense open subset of $\calX$ and replacing $S$ by a smooth dense open subset of $\pi(\calX)$, we may assume that $\calX$ dominates $S$. Shrinking once more, we may also assume that $T:=\iota_S(S)$ is a smooth locally closed subvariety of $\AAA_g$. Put $\calB:=\gothA_g|_T$. Then
\[
\calA=\calB\times_T S, \qquad \iota:\calA\longrightarrow\calB.
\]

We first note that $\iota(\calA_{<t})\subseteq\calB_{<t}$. Indeed, the level-$3$-structure ensures that all geometric generic endomorphisms of $\calB$ are defined over $K(T)$. Together with extension of generic endomorphisms over normal bases, this gives $\End(\calB/T)\xrightarrow{\sim}\End(\calA/S)$; see \cite[\S2.3]{Ge_IntersectI}. For any flat subgroup scheme $\mathcal H\subseteq\calA$, there is $\alpha\in\End(\calA/S)$ such that
\[
\mathcal H\subseteq\ker\alpha, \qquad \rdim\ker\alpha=\rdim\mathcal H.
\]
Descending $\alpha$ to $\alpha_T\in\End(\calB/T)$ shows that $\iota(\mathcal H)\subseteq\ker\alpha_T$, a flat subgroup scheme of the same relative dimension. This proves the claimed inclusion.

Set $\calY:=\iota(\calX)^{\mathrm{Zar}}\subseteq\calB$ and $f:=\iota|_{\calX}:\calX\longrightarrow\calY$. Every positive-dimensional irreducible component $F$ of a fiber of $f$ is a constant section over $\pi(F)$, because $\iota$ identifies each fiber of $\calA/S$ with the corresponding fiber of $\calB/T$. Consequently,
\[
\rdim F^{\textnormal{ws}}=0, \qquad \delta_v(F)=-\dim F<0,
\]
and hence $F\subseteq\calX^{\textnormal{deg}}(0)$. If $\dim\calX>\dim\calY$, every point of $\calX$ lies on such a positive-dimensional fiber component. Thus $\calX=\calX^{\textnormal{deg}}(0)$, and there is nothing to prove.

We may therefore assume that $f$ is generically finite. Choose a dense open subset $V\subseteq\calY$ such that, with $U:=f^{-1}(V)$, the morphism $f|_U:U\longrightarrow V$ is finite \'etale. We claim that
\[
U\cap f^{-1}\bigl(\calY^{\textnormal{deg}}(t)\bigr) \subseteq\calX^{\textnormal{deg}}(t).
\]
Let $x$ belong to the left-hand side, and choose a positive-dimensional irreducible subvariety $\calZ\subseteq\calY$ containing $f(x)$ with $\delta_v(\calZ)<t$. Let $\calZ'$ be the closure in $\calX$ of an irreducible component through $x$ of $f^{-1}(\calZ\cap V)$. Since $f|_U$ is finite \'etale, this component dominates $\calZ\cap V$, and $\dim\calZ'=\dim\calZ$. Pulling back the weakly special closure of $\calZ$ gives
\[
\rdim{(\calZ')^{\textnormal{ws}}} \leq\rdim{\calZ^{\textnormal{ws}}},
\]
by \cite[Lem.~3.3]{Ge_IntersectI}. Therefore
\[
\delta_v(\calZ') \leq\rdim{\calZ^{\textnormal{ws}}}-\dim\calZ =\delta_v(\calZ)<t.
\]
This proves the claim.

By the assumed case of \Cref{thm:main}, the set $(\calY\setminus\calY^{\textnormal{deg}}(t)) \cap\calB_{<t}$ is finite. The two inclusions proved above show that $(U\setminus\calX^{\textnormal{deg}}(t)) \cap\calA_{<t}$ maps into this finite set. Since $f|_U$ is finite, this intersection is finite as well. Finally, induction handles the proper closed subset $\calX\setminus U$, completing the proof.
\end{proof}

To sum up, we are left to show the following:

\begin{prop}\label{prop:main}
   Let $t\in \ZZ_{>0}$. Let $\AAA_g$ be the moduli space of principally polarized abelian varieties with symplectic level-$3$-structures over a number field $K$, and let $S\subseteq \AAA_g$ be a locally closed smooth subvariety. Let $\calA/S$ be the induced abelian scheme. For any subvariety $\calX$ of $\calA$, the intersection $(\calX \setminus \calX^\textnormal{deg}(t)) \cap \calA_{<t}$ is finite.
\end{prop}

\subsection{Proof of \Cref{prop:main}}

We may assume $\calX$ to be positive-dimensional. If $\calX^{\mathrm{deg}}(t)=\calX$, there is nothing to prove. Hence, we assume $\calX$ to be $t$-nondegenerate, and consequently we have $1\le t \le g-\dim\calX$.

Fix $x\in(\calX\setminus\calX^{\mathrm{deg}}(t))\cap\calA_{<t}$ and $s=\pi_g(x).$ Let $\Gamma \subseteq \End(\calA_s)$ be the specialized endomorphism algebra from $\End(\calA/S)$. The identity component of the $\Gamma$-special closure of $x$ is given by
\[
\langle x \rangle^0_\Gamma=\ker^0{\gamma} \qquad \textnormal{for some } \gamma \in \Gamma.
\]
The element $\gamma$ extends to an endomorphism $\phi \in \End(\calA/S)$, that determines an abelian subscheme $\calB/S$ of relative dimension $g_\calB < t$ such that $\calB_s=\langle x \rangle^0_\Gamma$. Call $M$ the number of connected components of $\langle x \rangle_\Gamma$. The connected component of $\langle x \rangle_\Gamma$ containing $x$ is of the type $q + \calB_s$, where $q$ is a torsion point in $\calA_s$ of order $n \le M$. We claim that $x$ is algebraic over $K$. Put $\calH:=[n]^{-1}\calB$. Then $\calH$ is a flat subgroup scheme of relative dimension $g_{\calB}$ containing $x$. Every positive-dimensional irreducible component $Y$ of $\calX\cap\calH$ satisfies
\[
\delta_v(Y)=\rdim Y^{\textnormal{ws}}-\dim Y
\leq g_{\calB}-\dim Y<t.
\]
Hence $Y\subseteq\calX^{\deg}(t)$, so $x$ is isolated
in the intersection $\calX\cap\calH$. Therefore
$x\in\calA(\bar K)$, and also $q\in\calA_s(\bar K)$.
For every $\sigma\in\operatorname{Gal}(\bar K/K)$, put
$x_\sigma:=\sigma(x)$ and $q_\sigma:=\sigma(q)$. Since $\calB$
is defined over $K$, we have
\[
x_\sigma-q_\sigma\in\calB_{\sigma(s)}
=(\ker^0\phi)_{\sigma(s)},\qquad
\operatorname{ord}(q_\sigma)=n.
\]
    
We now use the fundamental set $\gothF$ defined \Cref{subsec:o-minimality} and we denote by $p:\gothF \to \gothH_g$ the projection on the second component. The definable set $p(\gothF)\cap u_g^{-1}(S^{\mathrm{an}})$ has finitely many connected components $C_1,\ldots,C_N$. Each $C_j$ is contained in one connected component $\widetilde S_j$ of $u_g^{-1}(S^{\mathrm{an}})$; on it the relative rational representation of $\phi$ is the single matrix $\psi_j:=M_{r,\widetilde S_j}(\phi)$, as explained in \Cref{rmk:rrrMatrix}. Let $F_{j,\psi}$ be the restriction of the definable family $F(\calX)_\psi$ in \eqref{eq:defFamily} to points whose base belongs to $C_j$. These are finitely many definable families, uniformly parametrized by $\psi$.

For every distinct conjugate $x_\sigma$ choose $\tau_\sigma\in p(\gothF)$ above $\sigma(s)$, and lifts $\widetilde{x_\sigma},\widetilde{q_\sigma}\in\gothF$ over the \emph{same} $\tau_\sigma$. Denote their Betti vectors by $x'_\sigma,q'_\sigma\in[-1/2,1/2]^{2g}$. We may take $q'_\sigma\in n^{-1}\ZZ^{2g}$, and the standard Euclidean norm satisfies $|x'_\sigma-q'_\sigma|\leq\sqrt{2g}$. If $\tau_\sigma\in C_j$, the inclusion $x_\sigma-q_\sigma\in\calB_{\sigma(s)}$ gives
\[
x'_\sigma-q'_\sigma\in \ZZ^{2g}+\ker(\psi_j:\RR^{2g}\to\RR^{2g}).
\]
Moreover, by \Cref{lem:dictionary},
\[
\Lambda_j:=H_1(\calB_{\sigma(s)}^{\mathrm{an}},\ZZ) =\ker(\psi_j:\ZZ^{2g}\to\ZZ^{2g}),\qquad (\Lambda_j)_\RR=\ker(\psi_j:\RR^{2g}\to\RR^{2g}).
\]
Applying \Cref{lem:periodNorm} on each $\widetilde S_j$, and taking the maximum of its finitely many constants, yields $c>0$ and $\lambda'_\sigma\in\ZZ^{2g}$ such that
\[
x'_\sigma-q'_\sigma-\lambda'_\sigma\in(\Lambda_j)_\RR, \qquad |\lambda'_\sigma|\leq\sqrt{2g}+c \deg_\calL \calB.
\]
Considering the multiplicative height $H$ and the $\Gamma$-complexity $\Delta_\Gamma(x)$ of $x$, we obtain the inequalities
\[
H(q'_{\sigma})\leq \ord(q) \qquad \textnormal{and} \qquad \max\{\ord(q), \deg_\calL {\calB}\} \le \Delta_\Gamma(x).
\]
Notice that $w_\sigma:=q'_\sigma+\lambda'_\sigma\in\QQ^{2g}$. Thus, for some constant $C\geq1$ independent of $x,\sigma$, we have
\begin{equation}\label{eqn:naiveHeight}
H(w_\sigma)\leq \ord(q)(1+\sqrt{2g}+c \deg_\calL {\calB}) \leq C\Delta_\Gamma(x)^2.
\end{equation}
As $\psi_jw_\sigma=\psi_jx'_\sigma$, the pair $(w_\sigma,\widetilde{x_\sigma})$ belongs to $F_{j,\psi_j}$.

Since $x\notin \calX^{\mathrm{deg}}(t)$, by \cite[Thm.~8.4]{Ge_IntersectI} and \Cref{prop:LGO}, there exist $c_1, \epsilon_0 >0$ such that
\begin{equation}\label{eq:propMainLGO}
    [K(x):K] \ge c_1\Delta_\Gamma(x)^{\epsilon_0}.
\end{equation}

Set $T:=C\Delta_\Gamma(x)^2$, write $\operatorname{pr}_2$ for the second projection, and partition the pairs above into
\[
\Sigma_j:=\{(w_\sigma,\widetilde{x_\sigma}): \tau_\sigma\in C_j\} \subseteq F_{j,\psi_j}^{\sim}(T) \qquad(1\leq j\leq N).
\]
Since distinct conjugates have distinct lifts, for some $j$
\[
\#\operatorname{pr}_2(\Sigma_j) \geq\frac{c_1}{N}\Delta_\Gamma(x)^{\epsilon_0}.
\]

Put $\varepsilon:=\epsilon_0/4$, and let $c_{\mathrm{HP}}$ be the maximum of the constants in \Cref{prop:HP} for the finitely many families $F_j$ and this $\varepsilon$. Choose
\[
\ell>\max\left\{1, \left(\frac{Nc_{\mathrm{HP}}C^\varepsilon}{c_1} \right)^{2/\epsilon_0}\right\}.
\]
If $\Delta_\Gamma(x)\geq \ell$, then $\#\operatorname{pr}_2(\Sigma_j)> c_{\mathrm{HP}}T^\varepsilon$. Thus \Cref{prop:HP} gives a definable arc $\beta=(\beta_1,\beta_2):[0,1]\to F_{j,\psi_j}$ such that $\beta_1$ is semi-algebraic, $\beta_2$ is nonconstant, $\beta$ is real-analytic on $(0,1)$, and $\beta_2(0)=\widetilde{x_\sigma}$ for some conjugate $x_\sigma$ of $x$.

By \eqref{eq:defFamily}, for every $u\in[0,1]$ we have
\[
\psi_j\,\widetilde{\gothb}(\beta_2(u)) =\psi_j\,\beta_1(u).
\]
The image $\gamma:=\psi_j\beta_1([0,1])$ is semi-algebraic of dimension at most one. The arc $\beta_2$ lies over $\widetilde S_j$, so \Cref{prop:transcendence} applies directly to $\beta_2$ and the nonzero endomorphism $\phi$. Since $g_\phi=\rdim{\ker\phi}=g_\calB<t$ and $t$ is an integer, it yields
\[
x_\sigma\in u_{g,a}(\beta_2([0,1])) \subseteq\calX^{\mathrm{deg}}(g_\calB+1) \subseteq\calX^{\mathrm{deg}}(t).
\]
By \Cref{rmk:descent}, $\calX^{\mathrm{deg}}(t)$ is defined over $K$, contradicting the choice of $x$. Therefore $\Delta_\Gamma(x)<\ell$ uniformly.

Finally, only finitely many abelian subschemes $\calB\subseteq\calA$ have fiber degree $\deg_\calL {\calB}<\ell$.
% Indeed, specialize to one fixed complex fiber: by
% \Cref{lem:covolumeDegree}, bounded degree bounds
% the covolume of the period lattice. Minkowski's theorem
% then yields finitely many possible rational period
% subspaces and hence, by \Cref{lem:dictionary}, finitely
% many abelian subvarieties of that fiber. A relative
% abelian subscheme is determined by its period sublattice
% on one fiber. 
Also $n\leq\Delta_\Gamma(x)<\ell$. Thus every point under consideration belongs to one of finitely many flat group subschemes
\[
[n]^{-1}\calB=\calB+\calA[n], \qquad \rdim\calB<t.
\]
For each such subgroup, every positive-dimensional irreducible component $Y$ of its intersection with $\calX$ has $\delta_v(Y)\leq\rdim{\calB}-\dim Y<t$. Hence the points of $\calX\setminus\calX^{\mathrm{deg}}(t)$ in that intersection are isolated. A finite-type intersection has only finitely many isolated points, which proves the assertion.
\qed

\section{Applications}
In this final section, we give two applications of our main theorem.
The first concerns intersections with division groups generated by
sections. We formulate a \emph{relative Mordell--Lang conjecture},
motivated by the relative Manin--Mumford theorem~\cite{GH_RMM} and
the classical Mordell--Lang theorem~\cite{Faltings_GeneralCase},
and prove a case of it. The second application concerns injectivity
of specialization of a finitely generated group of sections.

\subsection{Relative Mordell--Lang}
\label{subsec:relativeML}

Let $S$ be a smooth quasi-projective variety
over a number field $K$, and let $\pi:\calA\to S$ be an abelian
scheme of relative dimension $g$. Let $\calX\subseteq\calA$ be an
irreducible closed subvariety dominating $S$, with geometrically irreducible generic fiber.

For a finitely generated subgroup $\Lambda\subseteq\calA(S)$,
denote its specialization at $s\in S$ by $\Lambda_s$. Define the division group by
\[
\Lambda_s^{\mathrm{div}}
:=\{x\in\calA_s:[N]x\in\Lambda_s
       \text{ for some }N\geq1\},
\]
and put
\[
\calM(\calX,\Lambda)
:=\bigcup_{s\in S(\bar K)}
\bigl(\calX_s(\bar K)\cap\Lambda_s^{\mathrm{div}}\bigr).
\]

The classical Mordell--Lang theorem concerns intersections with
a group of finite rank in a fixed abelian variety. In a family,
new relations may appear at special parameters even when the
generic intersection is empty. 

Assume that $\calX_{\bar\eta}$ is \emph{transverse with respect to
$\Lambda_{\bar\eta}^{\mathrm{div}}$}, that is,  for every
$a\in\Lambda_{\bar\eta}^{\mathrm{div}}$ and every strict abelian
subvariety $H\subsetneq\calA_{\bar\eta}$, we have
\begin{equation}\tag{T}\label{cond_transverse}
\calX_{\bar\eta}\not\subseteq a+H.
\end{equation}
We propose the following \emph{relative Mordell--Lang conjecture}.

\begin{conjecture}\label{conj:rml-sp}
If $\dim\calX<g$, then $\calM(\calX,\Lambda)$ is not Zariski
dense in $\calX$.
\end{conjecture}

When $S$ is a point, this follows from the classical
Mordell--Lang theorem for groups of finite rank. When $\Lambda=0$,
condition~\eqref{cond_transverse} excludes containment of the
geometric generic fiber in a strict torsion coset, and the
conjecture is the corresponding relative Manin--Mumford statement, established by Gao--Habegger~\cite{GH_RMM}.

\begin{prop}\label{prop:pink-implies-rml}
Pink's \Cref{conj:Pink} implies \Cref{conj:rml-sp}.
\end{prop}

\begin{proof}
Choose generators $P_1,\ldots,P_r$ of $\Lambda$ and form the graph
\[
j:\calX\rightarrow\calA_S^{r+1},\qquad
x\mapsto(x,P_1(\pi x),\ldots,P_r(\pi x)).
\]
Let $G_{\bar\eta}$ be the smallest algebraic subgroup containing
$j(\calX_{\bar\eta})$. Multiplying by an integer $m\geq1$ that
kills its component group, and making a finite base change and
shrinking $S$ if necessary, we obtain an abelian subscheme
$\calG\subseteq\calA_S^{r+1}$ and
$\calZ:=[m]j(\calX)\subseteq\calG$ whose geometric generic fiber
generates $\calG_{\bar\eta}$. The map $\calX\to\calZ$ is finite.
These reductions preserve the hypotheses, and nondensity
descends to the original base.

For $N\geq1$ and $u=(u_1,\ldots,u_r)\in\ZZ^r$, consider
\[
F_{N,u}:\calG\rightarrow\calA,\qquad
(z_0,\ldots,z_r)\mapsto[N]z_0-\sum_{i=1}^r[u_i]z_i.
\]
We claim that $F_{N,u}$ is surjective. Put
$H:=\operatorname{im}(F_{N,u,\bar\eta})$ and choose
$v\in\Lambda_{\bar\eta}^{\mathrm{div}}$ such that
\[
[N]v=\sum_{i=1}^r[u_i]P_i(\bar\eta).
\]
For every $x\in\calX_{\bar\eta}$, evaluation at $[m]j(x)$ gives
$[mN](x-v)\in H$. Consequently,
\[
\calX_{\bar\eta}\subseteq v+[mN]^{-1}H
=\bigcup_{\zeta\in\calA_{\bar\eta}[mN]}(v+\zeta+H).
\]
Geometric irreducibility puts $\calX_{\bar\eta}$ in one of these
translates. Since $v+\zeta\in\Lambda_{\bar\eta}^{\mathrm{div}}$,
condition~\eqref{cond_transverse} forces $H=\calA_{\bar\eta}$,
proving the claim.

The full kernels of the $F_{N,u}$ are therefore flat subgroup
schemes of relative codimension $g$. Every
$x\in\calM(\calX,\Lambda)$ gives a point $[m]j(x)$ in one of
these kernels. Since $\dim\calZ=\dim\calX<g$ and $\calZ$ is
not contained in a strict subgroup scheme of $\calG$, Pink's
\Cref{conj:Pink} gives nondensity of these points in $\calZ$.
The finite map $\calX\to\calZ$ then gives the assertion.
\end{proof}

\begin{defn}
    The abelian scheme $\calA/S$, or its generic fiber $\calA_\eta$, is said to have \emph{maximal variation} if the induced period map $S\rightarrow \AAA_g$ to the moduli space of abelian varieties is generically finite. 
\end{defn}

The following consequence of \Cref{thm:main} establishes
\Cref{conj:rml-sp} for section images in powers of an abelian
scheme under suitable variation hypotheses.

\begin{prop}\label{cor:rml-section-images}
Let $d:=\dim S>0$, and let $\calA/S$ be a polarized abelian
scheme of relative dimension $g$, with generic fiber $A$ over
$F:=K(S)$. Assume that every nontrivial simple abelian subvariety
of $A_{\bar F}$ has maximal variation. 

Let $\calB:=\calA_S^n$ for some $n\geq1$, let $\calX:=Q(S)$ for
$Q\in\calB(S)$, and let $\Lambda\subseteq\calB(S)$ be a
finitely generated subgroup. If $d<ng$ and
condition~\eqref{cond_transverse} holds for
$(\calB,\calX,\Lambda)$, then $\calM(\calX,\Lambda)$ is not
Zariski dense in $\calX$.
\end{prop}

\begin{proof}
Choose generators $P_1,\ldots,P_r$ of $\Lambda$. Similar to proof of \Cref{prop:pink-implies-rml}, after shrinking
$S$ and making a finite \'etale base change, we may choose an
integer $m\geq1$ and an abelian subscheme
$\calG\subseteq\calA_S^{r+1}$ such that
\[
R:=[m](Q,P_1,\ldots,P_r)\in\calG(S)
\]
and $R(\bar\eta)$ generates $\calG_{\bar\eta}$ as an algebraic
group. Put
$\calZ:=R(S)$ and $h:=\rdim\calG$.

For $N\geq1$ and $u=(u_1,\ldots,u_r)\in\ZZ^r$, define
\[
F_{N,u}:\calG\rightarrow\calA,\qquad
(z_0,\ldots,z_r)\mapsto[N]z_0-\sum_{i=1}^r[u_i]z_i.
\]
The surjectivity argument in the proof of
\Cref{prop:pink-implies-rml}, applied to this setting,
shows that every $F_{N,u}$ is surjective. Hence $h\geq ng$,
and its full kernel is a flat subgroup scheme of relative
dimension $h-ng$. Put
\[
t:=h-ng+1>0.
\]
Whenever $Q(s)\in\Lambda_s^{\mathrm{div}}$, the point $R(s)$
lies in one of these kernels. Therefore
\[
\{R(s):s\in S(\bar K),\ Q(s)\in\Lambda_s^{\mathrm{div}}\}
\subseteq\calZ\cap\calG_{<t}.
\]

We verify that $\calZ$ is $t$-nondegenerate using
\Cref{thm:Gao_criterion}, after the usual finite level and
isogeny reductions. Let $\calD$ be a quotient of $\calG$ of
positive relative dimension. Since
$\calG\subseteq\calA_S^{n(r+1)}$, every geometric simple
isogeny factor of $\calD$ is isogenous to an abelian subvariety
of $A_{\bar F}$. It follows that the moduli map of $\calD$ is generically finite.
As $\calZ$ is a section, its image under such a quotient
followed by the mixed period map has dimension
$d=\dim\calZ$. Thus no positive-dimensional quotient satisfies
the strict inequality in \Cref{thm:Gao_criterion}. For the
zero quotient, that inequality would read
\[
0<d-\max\{0,h-t\}=d-ng+1,
\]
which is impossible because $d<ng$.

By \Cref{thm:main}, $\calZ\cap\calG_{<t}$ is not Zariski dense
in $\calZ$. The isomorphism $\calX\rightarrow\calZ$ given
by $Q(s)\mapsto R(s)$ proves the assertion.
\end{proof}

% For example, suppose that $Q_1,\ldots,Q_n,P_1,\ldots,P_r\in\calA(S)$ are independent modulo torsion. Put $\Gamma:=\sum_i\mathbb ZP_i$, $\Lambda:=\Gamma^n\subseteq\calA^n_S(S)$, and $\calX:=(Q_1,\ldots,Q_n)(S)$. Then $b_\Lambda(\calX)=ng$: containment in a permitted translate of a proper abelian subvariety would, after applying a nonzero integer row vector and clearing denominators, give a nontrivial relation among the $Q_j$ and $P_i$. Therefore, if $d<ng$, the set
% \[
% \{s\in S(\bar K):Q_j(s)\in\Gamma_s^{\rm div} \text{ for every }1\leq j\leq n\}
% \]
% is not Zariski dense in $S$, with no restriction on $r$.

\subsection{Specialization of section groups}
Our second application concerns, in the same setup as above, the injectivity of specialization of $\Lambda\subseteq\calA(S)$.
Define
\[
\Sigma(\Lambda):=\{s\in S(\overbar K)\mid \spec_s:\Lambda \rightarrow \calA_s \text{ is not injective}.\}
\]

\begin{prop}
     If all nontrivial simple abelian subvarieties of $A_{\overbar{F}}$ have maximal variation and dimension $>\dim S$, then $\Sigma(\Lambda)$ is not Zariski dense in $S$.
\end{prop}
The maximal variation assumption on the simple components is equivalent to $A$ having no fixed part if $\dim S=1$.  The dimension assumption ensures the intersection is unlikely. 

\begin{proof}
Put $d:=\dim S$. The case $d=0$ is immediate. Since torsion
sections specialize injectively, we may assume that
$r:=\operatorname{rank}\Lambda>0$.

As in the proof of \cite[Thm.~9.4]{Ge_IntersectI}, after shrinking
$S$, making a finite base change, and taking suitable multiples
of a basis of $\Lambda$ modulo torsion, we obtain a section
$R=(P_1,\ldots,P_r)$ whose geometric generic point generates
an abelian subscheme $\calG\subseteq\calA_S^r$.
Put $\calZ:=R(S)$, $h:=\rdim\calG$, and $t:=h-d>0$.
The nondegeneracy argument there shows that $\calZ$ is
$t$-nondegenerate.

For $0\ne u\in\ZZ^r$, the homomorphism
\[
f_u:\calG\rightarrow\calA,\qquad
(z_1,\ldots,z_r)\mapsto\sum_{i=1}^r[u_i]z_i
\]
is nonzero by the independence of the $P_i$.
Its image has relative dimension greater than $d$ by
hypothesis, so its full kernel satisfies
\[
\rdim\ker f_u<h-d=t.
\]
If specialization fails to be injective at $s$, clearing
denominators and killing torsion gives
$R(s)\in\ker f_u$ for some $u\ne0$. Hence
\[
R(\Sigma(\Lambda))\subseteq\calZ\cap\calG_{<t}.
\]
By \Cref{thm:main}, the right-hand side is not Zariski dense
in $\calZ$. Since $R:S\to\calZ$ is an isomorphism, the
result follows. 
\end{proof}

%\appendix    

\bibliographystyle{abbrv}
\bibliography{arithmetic}

\end{document}